\documentclass[11pt,final]{amsart}

\usepackage{amssymb,amsmath,latexsym}
\usepackage{amsthm}
\usepackage{amsfonts}
\usepackage{enumerate}
\usepackage{booktabs}
\usepackage{mathrsfs}

\usepackage{hyperref}
\usepackage{xcolor}
\hypersetup{
    colorlinks, linkcolor={red!80!black},
    citecolor={blue!80!black}, urlcolor={blue}
} 
\usepackage[color]{showkeys}
 \definecolor{refkey}{gray}{.75}
 \definecolor{labelkey}{gray}{.75}

\usepackage{graphicx,pgf}
\usepackage{float}

\usepackage{fancyhdr}

\newtheorem{thm}{Theorem}[section]

\newtheorem{prop}{Proposition}[section]
\newtheorem{lemma}{Lemma}[section]
\newtheorem{rmk}{Remark}[section]

\def\ep{\epsilon}
\def\f{\frac}

\def\p{\partial}
\def\diag{\text{diag}}

\newcommand{\set}[1]{\left\{#1 \right\}}

\newcommand{\norm}[1]{\left\| #1 \right\| }
\newcommand{\inner}[2]{\langle #1, #2 \rangle}

\newcommand{\Rmnum}[1]{ \uppercase\expandafter{\romannumeral  #1}}

\newcommand{\mc}{\mathcal}

\numberwithin{equation}{section}
\numberwithin{figure}{section}

\let\oldtocsection=\tocsection
\let\oldtocsubsection=\tocsubsection
\let\oldtocsubsubsection=\tocsubsubsection
\renewcommand{\tocsection}[2]{\hspace{0em}\oldtocsection{#1}{#2}}
\renewcommand{\tocsubsection}[2]{\hspace{2em}\oldtocsubsection{#1}{#2}}
\renewcommand{\tocsubsubsection}[2]{\hspace{4.8em}\oldtocsubsubsection{#1}{#2}}

\begin{document}
\title{The Linearized 2D Inviscid Shallow Water Equations in a Rectangle: Boundary conditions and well-posedness}
\author{Aimin Huang}
\author{Roger Temam}
\address{The Institute for Scientific Computing and Applied Mathematics, Indiana University, 831 East Third Street, Rawles Hall, Bloomington, Indiana 47405, U.S.A.}
\email{AH:aimhuang@indiana.edu}
\email{RT:temam@indiana.edu}
\date{\today}

\begin{abstract}
We consider the linearized 2D inviscid shallow water equations in a rectangle. A set of boundary conditions is proposed which make these equations well-posed. Several different cases occur depending on the relative values of the reference velocities $(u_0,v_0)$ and reference height $\phi_0$ (sub- or super-critical flow at each part of the boundary).
\end{abstract}

\maketitle

\setcounter{tocdepth}{3}
\tableofcontents
\addtocontents{toc}{~\hfill\textbf{Page}\par}

\section{Introduction}

Motivated by the study of the inviscid linearized primitive equation (as a step towards the study of the full nonlinear equations), we consider in this article the linearized fully inviscid shallow water equations in a rectangle.
To the best of our knowledge, the well-posedness of the linearized inviscid shallow water equations in a rectangle has not yet been addressed. In particular, the existence of corners in the geometrical domain is a subject of mathematical concern since the seminal works \cite{Osh73, Osh74} which show the possible occurrence of major singularities in the corners for certain choices of the boundary conditions. The boundary conditions that we choose do not lead to singularities and allow us to develop results of existence and uniqueness of solution in suitable spaces to be described later on.

In this article, the well-posedness of the linearized inviscid shallow water equations is classically conducted using the linear semigroup theory. We write these equations as an abstract evolution equation in a Hilbert space:
\begin{equation*}
\f{\text{d}U}{\text{d}t} + AU = F,
\end{equation*}
with initial data $U(0)=U_0$ given, which means that we prescribe the boundary conditions at $t=0$ (see below). In this approach, the main difficulty is to show that $\inner{AU}{U}\geq 0$ for all $U$ in the domain $\mathcal{D}(A)$ of $A$ (and similarly $\inner{A^*U}{U}\geq 0,\,\forall U\in\mathcal{D}(A^*)$), where $A$ is the abstract associated operator acting in $L^2(\Omega)^3$ and $A^*$ is its adjoint. This result is achieved using a suitable approximation technique. 
{
Two main different cases occur depending on whether the linearized base flow with velocities $(u_0,v_0)$ and height $\phi_0$ is fully hyperbolic or partially hyperbolic and partially elliptic. The fully hyperbolic case corresponding to $u_0^2+v_0^2 > g\phi_0$ contains four sub-cases depending on whether the flow is sub- or super-critical on each side of the rectangle ($u_0^2, v_0^2 > $ or $<g\phi_0$). The elliptic-hyperbolic case corresponding to $u_0^2+v_0^2 < g\phi_0$ is fully subcritical $(u_0^2, v_0^2<g\phi_0)$;} furthermore the time independent part of the shallow water equations is partly elliptic in this case, but of course the full shallow water system remains hyperbolic. Boundary value problems of the Dirichlet or Neumann type appear in this case. Of course the presence of the corners also raises difficulties for these elliptic problems, but the desired regularity of the solutions of the corresponding Dirichlet-Neumann problem is achieved using results from \cite{Gri85} {which produce restrictions on the domain if other domains than rectangles were to be considered (see Remark 5.1 at the end)}.

The study of the Local Area Models (LAMs) in the atmosphere and oceans sciences, leads to initial and boundary value problems for the inviscid primitive equations in these domains, and we know since \cite{OS78} (see also \cite{TT03}) that these problems are complicated leading to nonlocal boundary value problems. As explained in e.g. \cite{WPT97} the choice of the boundary condition is important for the numerical simulations, as one wishes boundary conditions leading to well-posed problems (to avoid numerical explosion) and boundary conditions that are transparent, letting the waves move freely inside and outside the domain. In the perspective of a LAM imbedded in a large domain, any such set of conditions is acceptable. 

The linearized shallow water equations (SWE) that we consider read
\begin{equation}\label{eq:eq1.1}
\begin{cases}
u_t+u_0u_x + v_0u_y + g\phi_x -fv= 0, \\
v_t+u_0v_x + v_0v_y + g\phi_y + fu= 0, \\
\phi_t+u_0\phi_x + v_0\phi_y + \phi_0(u_x+v_y) = 0;
\end{cases}
\end{equation}
here $U=(u,v,\phi)^t,(x,y)\in \Omega:=(0,L_1)\times(0,L_2)$, $u_0,v_0,\phi_0$ are positive constants, and $g$ is the gravitational acceleration, $f$ is the Coriolis parameter.
As explained in e.g. \cite{OS78, TT03, RTT08b}, the inviscid primitive equations in a cube can be expanded in vertical modes expansion producing a set of coupled bi-dimensional equations, similar to coupled shallow water equations. For one linearized mode, the system that we obtain is very similar to \eqref{eq:eq1.1} as we show at the end of this article. With this respect the results that we obtain for the one-mode linearized primitive equations generalize in part the results from \cite{RTT08b} in which we assumed $v_0=0$; the full coupled system as considered in \cite{RTT08b} will be studied elsewhere.

As we know the literature on the shallow water equations is very vast, both on the theoretical and computational aspects, considering the viscous equations or the partly or totally inviscid equations and considering that the height is either always strictly positive or that it can vanish. See e.g. \cite{AB03, ABBKP, ABPS, BC01, BP91, BPSTT, SLTT}  on the computational side and see e.g. \cite{BD03, BDM07,  BN07,  Bre09, BR11,  HPT11, PT11,  KU82, KU88, KU89, Ore95, Ush83} on the theoretical side.
But as we said we did not find this problem studied in the literature. The general hyperbolic equations in a domain with corners are also considered in \cite{Tan78, KT80}, but the boundary conditions they imposed is different from ours, and the space they used is not the usual $L^2$ space which we use in this article.
{
Closest to our article is the book \cite{BS07} which systematically studies the initial and boundary value problems for linear hyperbolic equations but the spatial domain is smooth. Close to our article also are the articles \cite{Osh73,Osh74} which study the boundary value problem for hyperbolic equations in domains with corners; various singularities can occur near the corners of the domain. More precisely in \cite{Osh74}, the two dimensional shallow water equations are linearized around a constant flow as we do, and the boundary conditions are Dirichlet boundary conditions for linear combinations of the components as ours. Theorem~2.6 of \cite{Osh74} provides situations $(\tilde a)$ which produces a lack of uniqueness and $(\tilde b)$ which produces a lack of existence of solution. None of the examples provided in \cite{Osh74} is consistent with the boundary conditions that we propose. Note that \cite{Osh74} deals with a quarter plane $x>0,y>0$, but we believe that the situation is (or should be) the same as ours at the corners. Note also that our initial motivation related to LAMs is already expressed in \cite{Osh74} though in a more succinct way (middle of page 155). {Also as suggested in \cite{Osh74}, for any reasonable theory for hyperbolic boundary value problem, the solution will enjoy the finite speed propagation property so that at least in what we call below "the fully hyperbolic case", we could instead of a rectangle, have considered a single wedge or a more general polygonal domain. However this is not the case in what we call below the "elliptic-hyperbolic case". Also we preferred to work in a rectangular domain to stay close from the initial motivation in ocean and atmosphere science.}

Beyond our initial motivation related to the primitive equations and the LAMs in climate and weather studies, this work can be seen as a first step in studying the initial and boundary value problems for hyperbolic equations in non-smooth domains. Although the subject is very vast, we intend to extend this work in a few directions. Firstly, for linear equations, we intend to consider more general hyperbolic systems and possibly other convex polygons (pentagon, hexagon, etc). Secondly, we intend to study the nonlinear inviscid shallow water equations; see already the two cases considered in \cite{HPT12, HT12b}. 
Note that to study the nonlinear SWE we will probably need to study the equations linearized around a time dependent flow.
This is not the case in this article but this occurs in \cite{HPT12, HT12b}. 

A marking point of this study is that, in certain cases, the (time independent part of the) fully hyperbolic system contains a system which is partly hyperbolic and partly elliptic; this fact is probably general {and appears also in classical compressible fluid mechanics where the stationary flow can be subsonic, sonic or supersonic  that is governed by elliptic, parabolic or hyperbolic equations (see e.g. \cite{CF48})}. For this reason, this article stands at the interface between hyperbolic equations and parabolic-elliptic equations, but the methods used are rather those of parabolic and elliptic equations; this includes a regularization technique of L. H\"ormander \cite{Hor65} and regularity results by P. Grisvard \cite{Gri85} for solutions of elliptic equations in polygonal domains. 

This article is organized as follows. {At the end of this introductory section, we study the associated stationary system and find that the stationary system can be decoupled into some simple equations. 
Then in the technical Section \ref{sec-trace-density}, we present a useful trace theorem and derive various density theorems, density of certain smooth functions in certain function spaces. 
It prepares to the necessary density results needed  in Section \ref{sec-hyperbolic} in which we study the stationary shallow water equations operator in the fully hyperbolic case when $u_0^2+v_0^2>g\phi_0$.
Section \ref{sec-elliptic-hyperbolic} is devoted to the study of the stationary shallow water equations operator in the elliptic-hyperbolic case when $u_0^2+v_0^2<g\phi_0$, where we first prove a regularity result for an associated elliptic problem.
}
In Section \ref{sec-full-system} we state the general existence and uniqueness theorem proved by application of the Hille-Phillips-Yoshida theorem, which covers all cases (all respective values of $u_0,v_0,\phi_0$ indicated above). We also make the connection with the equations for one mode of the linearized primitive equations and interpret our earlier result for that case and briefly show that the results can be extended to the case of non-homogeneous boundary conditions.
}

\textbf{The associated stationary system.} 
In order to study the evolution equations \eqref{eq:eq1.1}, we introduce the stationary 2D shallow water equations operator:
\begin{equation}\label{eq:eq2.1}
\mathcal{A}U :=\mathcal{A}
\begin{pmatrix}
u \\ v \\ \phi
\end{pmatrix} =
\begin{pmatrix}
u_0u_x + v_0u_y + g\phi_x  \\
u_0v_x + v_0v_y + g\phi_y  \\
u_0\phi_x + v_0\phi_y + \phi_0(u_x+v_y) 
\end{pmatrix}=: \mathcal{E}_1U_x + \mathcal{E}_2U_y,
\end{equation}
where
\begin{equation*}
\mathcal{E}_1=\begin{pmatrix}
u_0&0&g\\
0&u_0&0\\
\phi_0&0&u_0
\end{pmatrix},\hspace{6pt}
\mathcal{E}_2=\begin{pmatrix}
v_0&0&0\\
0&v_0&g\\
0&\phi_0&v_0
\end{pmatrix}.
\end{equation*}
We first aim to determine boundary conditions which are suitable for the stationary system
\begin{equation}\label{eq:eq2.2}
\mathcal{A}U = \mathcal{E}_1U_x + \mathcal{E}_2U_y = F,
\end{equation}
{and these boundary conditions should be also suitable for the evolution equations \eqref{eq:eq1.1}. Note that $\mathcal{E}_1$, $\mathcal{E}_2$ admit a symmetrizer $S_0=\text{diag}(1,1,g/\phi_0)$, i.e. $S_0\mathcal{E}_1, S_0\mathcal{E}_2$ are both symmetric. }
In the following, we assume that $u_0>0,v_0>0$. The case where $u_0$ and, or $v_0$ are negative can be treated in a similar manner; we do not study the non-generic case where one or both of $u_0,v_0$ vanish. By its physical meaning (height), $\phi_0\geq 0$ and we do not study the case where $\phi_0=0$. Similarly, we only study the generic cases where $u_0^2\ne g\phi_0$, $v_0^2\ne g\phi_0$ and $u_0^2+v_0^2\neq g\phi_0$. 
{There are two cases to study regarding the sign of $\Delta=u_0^2+v_0^2-g\phi_0$. When $\Delta>0$, the system \eqref{eq:eq2.2} is fully hyperbolic and when $\Delta <0$, the system \eqref{eq:eq2.2} is partially hyperbolic and partially elliptic.}	

\vspace{6pt}

{
\textsc{The fully hyperbolic case.} We first consider the fully hyperbolic case when $\Delta=u_0^2+v_0^2-g\phi_0$ is positive, and 
we set $\kappa_0=\sqrt{g(u_0^2+v_0^2-g\phi_0)/\phi_0}$.
The system \eqref{eq:eq2.2} can be viewed as a one-dimensional hyperbolic problem by treating either the $x$- or $y$-direction as the time-like direction and let us choose the $y$-direction. Then, multiplying both sides of \eqref{eq:eq2.2} by $\mathcal{E}_2^{-1}$ gives
}
\begin{equation}\label{eq:eq2.3}
U_y + \mathcal{E}_2^{-1}\mathcal{E}_1U_x = \mathcal{E}_2^{-1}F.
\end{equation}
Directly or on using Matlab, we compute
\begin{equation}
\mathcal{E}_2^{-1}\mathcal{E}_1=\f{1}{v_0^2-g\phi_0}\begin{pmatrix}
&\f{u_0}{v_0}(v_0^2-g\phi_0)&0&\f{g}{v_0}(v_0^2-g\phi_0) \\
&-g\phi_0 &u_0v_0 &-gu_0 \\
&v_0\phi_0 &-u_0\phi_0&u_0v_0
\end{pmatrix},
\end{equation}
and write
\begin{equation}\label{eq:eq2.5}
\hat P^{-1}\cdotp \mathcal{E}_2^{-1}\mathcal{E}_1\cdotp \hat P = \text{diag}(\lambda_1, \lambda_2, \lambda_3),
\end{equation}
where $\hat P$ has a complicated expression, whereas
\begin{equation*}
\hat P^{-1}=\begin{pmatrix}
&\f{v_0}{2\kappa_0} &-\f{u_0}{2\kappa_0} &\f{1}{2} \\
&-\f{v_0}{2\kappa_0} &\f{u_0}{2\kappa_0} &\f{1}{2} \\
&\f{u_0v_0}{u_0^2+v_0^2}&\f{v_0^2}{u_0^2+v_0^2}&\f{gv_0}{u_0^2+v_0^2}
\end{pmatrix},
\end{equation*}
and
\begin{equation}\label{eq:eq2.4}
\lambda_1=\f{u_0v_0+\phi_0\kappa_0}{v_0^2-g\phi_0},\hspace{6pt}
\lambda_2=\f{u_0v_0-\phi_0\kappa_0}{v_0^2-g\phi_0},\hspace{6pt}
\lambda_3=\f{u_0}{v_0}.
\end{equation}

{
The remarkable point we find is that the matrix $\hat P$ can congruently diagonalize $S_0\mathcal{E}_1$ and $S_0\mathcal{E}_2$ simultaneously. Here, we use a different version of the matrix $\hat P$, which we call ${P}$, where
$$
{P}^{-1}=\begin{pmatrix}
v_0&-u_0&\kappa_0 \\
v_0&-u_0&-\kappa_0 \\
u_0&v_0&g
\end{pmatrix}=\begin{pmatrix}
2\kappa_0&0&0\\
0&-2\kappa_0&0\\
0&0&(u_0^2 + v_0^2) / v_0
\end{pmatrix}\cdot \hat P^{-1}.
$$
The matrix $P$ can also congruently diagonalize $S_0\mathcal{E}_1$ and $S_0\mathcal{E}_2$ simultaneously, and we obtain by using Matlab or by direct calculations:
\begin{equation}\label{eqe.1}
\begin{split}
{P}^t S_0\mathcal{E}_1{P} &={\diag}(\f{u_0\kappa_0 + gv_0}{2(u_0^2+v_0^2)\kappa_0}, \f{u_0\kappa_0 - gv_0}{2(u_0^2+v_0^2)\kappa_0}, \f{u_0}{u_0^2+v_0^2})=:\diag(a_1,a_2,a_3),\\
{P}^t S_0\mathcal{E}_2{P} &={\diag}(\f{v_0\kappa_0 - gu_0}{2(u_0^2+v_0^2)\kappa_0}, \f{v_0\kappa_0 + gu_0}{2(u_0^2+v_0^2)\kappa_0}, \f{v_0}{u_0^2+v_0^2})=:\diag(b_1,b_2,b_3).
\end{split}
\end{equation}

Now we introduce the new variables
\begin{equation}\label{eq:eq2.6}
\begin{pmatrix}
\xi \\ \eta \\ \zeta
\end{pmatrix}
= P^{-1}\begin{pmatrix}
u \\ v \\ \phi
\end{pmatrix}
=\begin{pmatrix}
v_0u - u_0v + \kappa_0 \phi \\
v_0u - u_0v - \kappa_0 \phi \\
u_0u + v_0v + g\phi
\end{pmatrix};
\end{equation}
multiplying both sides of \eqref{eq:eq2.2} by $P^tS_0$ gives
\begin{equation}\label{eq:eq2.7}
\begin{pmatrix}
a_1& 0 &0 \\
0& a_2 & 0  \\
0&0& a_3 
\end{pmatrix}\begin{pmatrix}
\xi \\ \eta \\ \zeta
\end{pmatrix}_x
+\begin{pmatrix}
b_1& 0 &0 \\
0& b_2 & 0  \\
0&0& b_3
\end{pmatrix}
\begin{pmatrix}
\xi \\ \eta \\ \zeta
\end{pmatrix}_y
= P^t S_0 F,
\end{equation}
which is a fully decoupled system. 
In order to determine the boundary conditions for \eqref{eq:eq2.7} (and for \eqref{eq:eq2.2}), we choose the boundary conditions for the incoming characteristics according to the signs of $a_i$ and $b_i$ ($i=1,2,3$). We will further study \eqref{eq:eq2.7} in Section \ref{sec-hyperbolic}.
}

\vspace{6pt}

{
\textsc{The elliptic-hyperbolic case.} We now consider the elliptic-hyperbolic case when $\Delta=u_0^2+v_0^2-g\phi_0$ is negative and we set $\kappa_1=\sqrt{g(g\phi_0-u_0^2-v_0^2)/\phi_0}$. 
We notice that \eqref{eq:eq2.7} still holds but the $a_i$ and $b_i$ ($i=1,2$) are now complex numbers, which indicates that there are some elliptic modes hidden in \eqref{eq:eq2.2}. 
}
In order to avoid computations with complex numbers and to take advantage of the regularity results valid for elliptic systems (see Subsection \ref{sec-regularity}), we now define $\xi,\eta,\zeta$ and $P$ by setting
\begin{equation}\label{eqe.5}
P^{-1}=\begin{pmatrix}
v_0&-u_0&0 \\
0&0&\kappa_1 \\
u_0&v_0&g
\end{pmatrix},\hspace{6pt}
\begin{pmatrix}
\xi \\ \eta \\ \zeta
\end{pmatrix}
=P^{-1}\begin{pmatrix}
u \\ v \\ \phi
\end{pmatrix}
=\begin{pmatrix}
v_0u - u_0v \\
 \kappa_1 \phi\\
u_0u + v_0v + g\phi
\end{pmatrix};
\end{equation}
{
and we have the following diagonalization result corresponding to \eqref{eqe.1}:
\begin{equation}\label{eqe.2}
\begin{split}
{P}^t S_0\mathcal{E}_1{P} =\f{1}{u_0^2+v_0^2}\begin{pmatrix}
u_0&\f{gv_0}{\kappa_1}&0\\
\f{gv_0}{\kappa_1}&-u_0&0\\
0&0&u_0
\end{pmatrix},\;
{P}^t S_0\mathcal{E}_2{P} =\f{1}{u_0^2+v_0^2}\begin{pmatrix}
v_0&-\f{gu_0}{\kappa_1}&0\\
-\f{gu_0}{\kappa_1}&-v_0&0\\
0&0&v_0
\end{pmatrix}.
\end{split}
\end{equation}
}

Then multiplying to the left both sides of \eqref{eq:eq2.2} by $P^tS_0$ gives
\begin{equation}\label{eq:eq2.8}
\f{1}{u_0^2+v_0^2}\begin{pmatrix}
u_0&\f{gv_0}{\kappa_1}&0\\
\f{gv_0}{\kappa_1}&-u_0&0\\
0&0&u_0
\end{pmatrix}
\begin{pmatrix}
\xi \\ \eta \\ \zeta
\end{pmatrix}_x
+
\f{1}{u_0^2+v_0^2}\begin{pmatrix}
v_0&-\f{gu_0}{\kappa_1}&0\\
-\f{gu_0}{\kappa_1}&-v_0&0\\
0&0&v_0
\end{pmatrix}
\begin{pmatrix}
\xi \\ \eta \\ \zeta
\end{pmatrix}_y
=P^tS_0F.
\end{equation}

We see from \eqref{eq:eq2.8} that the equation satisfied by $\zeta$ is decoupled, and is a transport (hyperbolic) equation.
The equations satisfied by $\xi$ and $\eta$ read
\begin{equation}\label{eq:eq2.9}
\begin{pmatrix}
u_0&\f{gv_0}{\kappa_1}\\
\f{gv_0}{\kappa_1}&-u_0\\
\end{pmatrix}
\begin{pmatrix}
\xi \\ \eta
\end{pmatrix}_x
+
\begin{pmatrix}
v_0&-\f{gu_0}{\kappa_1}\\
-\f{gu_0}{\kappa_1}&-v_0\\
\end{pmatrix}
\begin{pmatrix}
\xi \\ \eta 
\end{pmatrix}_y
=\begin{pmatrix}
(u_0^2+v_0^2)f_1 \\ (u_0^2+v_0^2)f_2
\end{pmatrix},
\end{equation}
where $f_1,f_2$ are the first two components of $P^tS_0F$. The equations \eqref{eq:eq2.9} are elliptic since the four constants $u_0, v_0,gv_0/\kappa_1,-gu_0/\kappa_1$ satisfy condition \eqref{eq:con6.1} (see Subsection \ref{sec-regularity}). We will further study \eqref{eq:eq2.8} in Subsection \ref{sec-mixed-subcritical}.

\section{Trace and density theorems}\label{sec-trace-density}

\subsection{A trace theorem}Since we have to assign boundary conditions on $\p\Omega$, we first need to make sure that the desired traces at the boundary make sense {for the functional setting that we will use}. Thus we consider the space:
\begin{equation*}
\mathcal{X}(\Omega)=\{ U\in H=L^2(\Omega)^3, \mathcal{A}U\in H\},
\end{equation*}
endowed with its natural Hilbert norm $(\norm{U}_H^2 + \norm{\mathcal{A}U}_H^2)^{1/2}$. 
Then we have the following trace result.
\begin{prop}\label{prop2.1}
If $U=(u,v,\phi)^t\in\mathcal{X}(\Omega)$, then the traces of $U$ are defined on all of $\p\Omega$, i.e. the traces of $U$ are defined at $x=0,L_1$, and $y=0,L_2$, and they belong to the respective spaces $H_y^{-1}(0,L_2)$ and $H_x^{-1}(0,L_1)$. Furthermore the trace operators are linear continuous in the corresponding spaces, e.g., $U\in\mathcal{X}(\Omega)\rightarrow U|_{x=0}$ is continuous from $\mathcal{X}(\Omega)$ into $H_y^{-1}(0,L_2)^3$.
\end{prop}
\begin{proof}
Since $U\in L^2(\Omega)^3=L_x^2(0,L_1;L_y^2(0,L_2)^3)$, we see that $U_y=\p U/\p y$ belongs to $L_x^2(0,L_1;H_y^{-1}(0,L_2)^3)$. Since $\mathcal{A}U\in L^2(\Omega)^3$, that is 
\begin{equation}\label{eq:eq2.11}
\begin{cases}
u_0u_x + v_0u_y + g\phi_x &\in L_x^2(0,L_1;L_y^2(0,L_2)), \\
u_0v_x + v_0v_y + g\phi_y  &\in L_x^2(0,L_1;L_y^2(0,L_2)), \\
u_0\phi_x + v_0\phi_y + \phi_0(u_x+v_y) &\in L_x^2(0,L_1;L_y^2(0,L_2)),
\end{cases}
\end{equation}
we infer from $\eqref{eq:eq2.11}_{1,3}$ that
\begin{equation}\label{eq:eq2.12}
\begin{cases}
u_0u_x + g\phi_x &\in L_x^2(0,L_1;H_y^{-1}(0,L_2)), \\
u_0\phi_x+\phi_0 u_x &\in L_x^2(0,L_1;H_y^{-1}(0,L_2)).
\end{cases}
\end{equation}
Noticing that the determinant of the coefficient matrix in \eqref{eq:eq2.12} is nonzero since we exclude the non-generic cases where $u_0^2=g\phi_0$, we obtain that $u_x,\phi_x\in L_x^2(0,L_1;H_y^{-1}(0,L_2))$. We also find that $v_x\in L_x^2(0,L_1;H_y^{-1}(0,L_2))$ from $\eqref{eq:eq2.11}_2$. Therefore, we have
\begin{equation*}
u_x,v_x,\phi_x\in L_x^2(0,L_1;H_y^{-1}(0,L_2)),
\end{equation*}
and in combination with $u,v,\phi\in L_x^2(0,L_1;L_y^2(0,L_2))$, we obtain that
\begin{equation}
u,v,\phi \in \mathcal{C}_x([0, L_1]; H_y^{-1}(0,L_2)),
\end{equation}
which shows that the traces of $u,v,\phi$ are well-defined at $x=0$ and $L_1$, and belong to $H_y^{-1}(0,L_2)$. The continuity of the corresponding mappings is easy. The proof for the traces at $y=0$ and $L_2$ is similar.
\end{proof}

\subsection{The density theorems}\label{sec-density}
{
In this subsection, we establish general density theorems regarding functions defined on the domain $\Omega=(0,L_1)\times(0,L_2)$. These theorems will be needed for proving later on that $-A$ is the infinitesimal generator of a semigroup of contraction, {where $A$ is the abstract operator associated with $\mathcal{A}$ and with certain boundary conditions to be chosen in each case separately.}

For $\lambda$ fixed, $\lambda\in \mathbb{R}, \lambda\neq 0$, we set $T\theta=\theta_y+\lambda \theta_x$, and introduce the function space
\begin{equation*}
\mathcal{X}_1(\Omega)=\{ \theta\in L^2(\Omega), T\theta=\theta_y+\lambda \theta_x\in L^2(\Omega) \}.
\end{equation*}
We need to show that the smooth functions are dense in $\mathcal{X}_1(\Omega)$. Later on we will prove more involved density theorems, showing that if $u\in\mathcal{X}_1(\Omega)$ vanishes on certain parts of $\p\Omega$, then $u$ can be approximated in $\mathcal{X}_1(\Omega)$ by smooth functions, vanishing on the same parts of the boundary. 
We first have the following.
\begin{prop}\label{prop1}
$\mathcal{C}^\infty(\overline\Omega)\cap\mathcal{X}_1(\Omega)$ is dense in $\mathcal{X}_1(\Omega)$.
\end{prop}
The proof of Proposition \ref{prop1} is classically conducted by using the method of partition of unity, hence we omit it here.

\begin{rmk}
Proposition \ref{prop1} is also valid with $T$ replaced by any first order differential operator with constant coefficients.
\end{rmk}

We are now going to prove the density theorems involving the boundary $\p\Omega$, and we will successively consider the scalar case and the vector case. 
Since we have to assign boundary conditions on $\p\Omega$, we first need to make sure that the desired traces at the boundary make sense. We actually have the following result proven exactly as in Proposition \ref{prop2.1}. 
\begin{prop}\label{M-propb.2}
If $\theta\in \mathcal{X}_1(\Omega)$, then the traces of $\theta$ are defined on all of $\p\Omega$, i.e. the traces of $u$ are defined at $x=0,L_1$, and $y=0,L_2$, and they belong to the respective spaces $H_y^{-1}(0,L_2)$ and $H_x^{-1}(0,L_1)$. Furthermore the trace operators are linear continuous in the corresponding spaces, e.g., $\theta\in\mathcal{X}_1(\Omega)\rightarrow \theta|_{x=0}$ is continuous from $\mathcal{X}_1(\Omega)$ into $H_y^{-1}(0,L_2)$.
\end{prop}

Here and throughout this article we denote by $\Gamma_W,\Gamma_E,\Gamma_S,\Gamma_N$ the boundaries $x=0,x=L_1,y=0,y=L_2$ respectively, and let $\Gamma$ be any union of the sets $\Gamma_W,\Gamma_E,\Gamma_S,\Gamma_N$. For any function $v$ defined on $\Omega$, here and again in the following we denote by $\tilde v$ the function equal to $v$ in $\Omega$ and to $0$ in $\mathbb{R}^2\backslash\Omega$. 
In the scalar case, we introduce the function spaces:
\begin{equation}\label{eq:eq01}
\mathcal{X}_\Gamma(\Omega)=\{ \theta\in L^2(\Omega), T\theta=\theta_y+\lambda \theta_x\in L^2(\Omega), \theta|_\Gamma = 0 \},
\end{equation}
\begin{equation*}
\mathcal{V}_\Gamma(\Omega) = \{ \theta\in \mathcal{C}^\infty(\overline{\Omega}), 
\text{ and $\theta$ vanishes in a neighborhood of } \Gamma\}.
\end{equation*}

We then state the density theorem:
\begin{thm}\label{thm1}
Suppose that $\Gamma=\Gamma_W\cup\Gamma_S$ and $\lambda>0$. Then 
\begin{equation*}
\mathcal{V}_\Gamma(\Omega)\cap\mathcal{X}_\Gamma(\Omega) \text{ is dense in } \mathcal{X}_\Gamma(\Omega).
\end{equation*}
\end{thm}
\begin{proof}
Let $\rho(x,y)$ be a mollifier such that $\rho(x,y)\geq 0, \int \rho \text{d}x\text{d}y=1$ and $\rho$ has compact support in $\{0<\f{1}{2}x<y<2x\}$. 
For $\theta\in\mc X_\Gamma(\Omega)$, we observe that 
\begin{equation}\label{eq123}
T\tilde\theta = \widetilde{T\theta} + \mu,
\end{equation}
where $\mu$ is a measure supported on $\{x=L_1\}\cup\{y=L_2\}$. We set $\tilde \theta_\ep=\rho_\ep*\tilde \theta$, and mollifying \eqref{eq123} with $\rho$ (see \cite{Hor65}) gives
\begin{equation}\label{eq:eq07}
 T\tilde \theta_\ep=\rho_\ep*\widetilde{T\theta} + \rho_\ep*\mu,
 \end{equation}
where $\rho_\ep*\mu$ is supported in $\Omega^c$ by the choice of $\rho$. Hence, restricting \eqref{eq:eq07} to $\Omega$ implies that:
\begin{equation}\label{eq07extra1}
 T\tilde \theta_\ep\big{|}_{\Omega}=(\rho_\ep*\widetilde{T\theta})\big{|}_{\Omega} \rightarrow T\theta, \text{ in } L^2(\Omega) \text{ , as }\ep\rightarrow 0,
\end{equation}
Then, as $\ep\rightarrow 0$, $\tilde \theta_\ep\rightarrow\tilde \theta$ in $L^2(\mathbb{R}^2)$, and
\begin{equation}
\begin{cases}
\tilde \theta_\ep\big{|}_{\Omega} \rightarrow \theta, \text{ in } L^2(\Omega), \\
T\tilde \theta_\ep\big{|}_{\Omega}\rightarrow T\theta,\text{ in } L^2(\Omega).
\end{cases}
\end{equation}
Finally, $\tilde \theta_\ep\big{|}_{\Omega}$ vanishes in a neighborhood of $\Gamma$ since the support of $\tilde \theta_\ep\big{|}_{\Omega}$ is away from $\Gamma$ by the choice of $\rho$. We thus completed the proof.
\end{proof}

}
\begin{rmk}\label{rmk2}
Looking carefully at the proof of Theorem \ref{thm1}, {we see that Theorem \ref{thm1} is also valid when $\Gamma$ is made of two contiguous sides of $\p\Omega$,} that is in the following three cases:
\begin{equation}
\begin{cases}
\Gamma=\Gamma_E\cup\Gamma_N \text{ and } \lambda >0, \\
\Gamma=\Gamma_W\cup\Gamma_N \text{ and } \lambda <0, \\
\Gamma=\Gamma_E\cup\Gamma_S \text{ and } \lambda <0,
\end{cases}
\end{equation}
provided only we choose properly the support of the mollifier. This remark will be useful for Theorem \ref{thm2} and when we study the adjoint $A^*$ of $A$ later on.
\end{rmk}
We now turn to generalizing Theorem \ref{thm1} to the vector case.
\begin{thm}\label{thm2}
Let $n$ be a positive integer, and suppose that $\Lambda=diag(\lambda_1,\cdots,\lambda_n)$ is a nonsingular diagonal matrix, and $M_1,M_2,Q$ are arbitrary nonsingular matrices satisfying $Q^{-1}\Lambda Q=M_2^{-1}M_1$. Set $\vec{\Gamma}=(\Gamma^1,\cdots, \Gamma^n)$ where 
\begin{equation*}
\Gamma^i=\begin{cases}
\Gamma_W\cup\Gamma_S\text{ or }\Gamma_E\cup\Gamma_N, \text{ if }\lambda_i>0,\\
\Gamma_W\cup\Gamma_N\text{ or }\Gamma_E\cup\Gamma_S, \text{ if }\lambda_i<0;
\end{cases}
\end{equation*}
and introduce the function spaces in the vector case:
\begin{equation*}
\mathcal{X}_{\vec{\Gamma}}(\Omega)=\{ \Theta=(\theta_1,\cdots,\theta_n)\in L^2(\Omega)^n, M_1\Theta_x+M_2\Theta_y\in L^2(\Omega)^n, Q\Theta|_{\vec{\Gamma}} = 0 \},
\end{equation*}
\begin{equation*}
\mathcal{V}_{\vec{\Gamma}}(\Omega) = \{ \Theta\in \mathcal{C}^\infty(\overline{\Omega}), 
\text{ and $Q\Theta$ vanishes in a neighborhood of } \vec{\Gamma}\}.
\end{equation*}
Then we have 
\begin{equation*}
\mathcal{V}_{\vec{\Gamma}}(\Omega)\cap\mathcal{X}_{\vec{\Gamma}}(\Omega) \text{ is dense in }\mathcal{X}_{\vec{\Gamma}}(\Omega).
\end{equation*}
\end{thm}
\begin{proof}
Setting $\widehat{\Theta}=Q\Theta$, we first have $\widehat{\Theta}_y +\Lambda \widehat{\Theta}_x=QM_2^{-1}(M_1\Theta_x+M_2\Theta_y)$, and then applying Theorem \ref{thm1} or Remark \ref{rmk2} to each component of $\widehat{\Theta}$ and transforming back to $\Theta$, we obtain the result for Theorem \ref{thm2}.
\end{proof}

\section{The fully hyperbolic case}\label{sec-hyperbolic}
{
In this section, we study the stationary 2D shallow water equations operator for the fully hyperbolic case where $\Delta=u_0^2+v_0^2-g\phi_0$ is positive. We have four sub-cases to consider depending on the signs of $u_0^2 - g\phi_0$ and $v_0^2-g\phi_0$, and we call them the supercritical case, the mixed hyperbolic case (two sub-cases) and the fully hyperbolic subcritical case; we will study these cases in the following subsections. We should bear in mind the fully decoupled system \eqref{eq:eq2.2} in this case, and it is easy to see that
\[
a_1,\,a_3,\, b_2,\,b_3 >0,
\]
since $u_0,v_0$ and $\phi_0$ are all positive, and the signs of $a_2$ and $b_1$ depend on the signs of $u_0^2 - g\phi_0$ and $v_0^2-g\phi_0$ respectively.
}

\subsection{The supercritical case}\label{sec-supercritical}
{The supercritical (fully supersonic) case corresponds to the case where:}
\begin{equation}\label{eq:eq3.0}
 u_0^2 > g\phi_0,\hspace{6pt} v_0^2 > g\phi_0. 
\end{equation}

{Under assumption \eqref{eq:eq3.0}, we see that $a_2$ and $b_1$ are both positive and we conclude that
\[
a_i,\, b_i > 0,\quad\quad \forall\,i=1,2,3.
\]
According to the general hyperbolic theory, we choose the boundary conditions for the incoming characteristics, hence, we choose the boundary conditions for $(\xi,\eta,\zeta)$ at $x=0$ and $y=0$ since $a_i$ and $b_i$ ($i=1,2,3$) are all positive, which is equivalent to choose the boundary conditions for $U=(u,v,\phi)$ at $x=0$ and $y=0$, i.e.}
\begin{equation}\label{eq:eq3.1}
\begin{cases}
u = v = \phi = 0, \text{ on } \Gamma_W=\{x=0\}, \\
u = v = \phi = 0, \text{ on } \Gamma_S=\{y=0\}.
\end{cases}
\end{equation}
We then define the unbounded operator $A$ on $H=L^2(\Omega)^3$, with $AU=\mathcal{A}U,\,\forall U\in\mathcal{D}(A)$ and
\begin{equation*}
\mathcal{D}(A) = \{ U\in H = L^2(\Omega)^3, AU\in H, \text{ and } U\text{ satisfies }\eqref{eq:eq3.1} \}.
\end{equation*}
 We also introduce the corresponding density boundary conditions:
\begin{equation}\label{eq:eq3.1prime}
U\text{ vanishes in a neighborhood of }\Gamma_W\cup\Gamma_S,
\end{equation}
and the function space:
\begin{equation*}
\mathcal{V}(\Omega) = \{ U\in \mathcal{C}^\infty(\overline{\Omega})^3, \text{ and } U\text{ satisfies }\eqref{eq:eq3.1prime}
\}.
\end{equation*}
Setting $\vec{\Gamma}=(\Gamma_W\cup\Gamma_S,\Gamma_W\cup\Gamma_S,\Gamma_W\cup\Gamma_S)$, we have that $\hat P^{-1}U|_{\vec{\Gamma}}=0$ is equivalent to \eqref{eq:eq3.1} and $\hat P^{-1}U$ vanishing in a neighborhood of $\vec{\Gamma}$ is the same as \eqref{eq:eq3.1prime}, where $\hat P$ is as in \eqref{eq:eq2.5}.
Then we obtain the following result from Theorem \ref{thm2} with $Q=\hat P^{-1}, M_1=\mathcal{E}_1, M_2=\mathcal{E}_2, \Lambda=\diag(\lambda_1,\lambda_2,\lambda_3)$.
\begin{lemma}\label{lem3.1}
$\mathcal{V}(\Omega)\cap\mathcal{D}(A)$ is dense in $\mathcal{D}(A)$.
\end{lemma}

\subsubsection{Positivity of $A$ and its adjoint $A^*$}\label{sec-supercritical-sub1}
We endow the space $H=L^2(\Omega)^3$ with the Hilbert scalar product and norm:
\begin{equation*}
\inner{U}{\overline{U}}_H = \int_\Omega(u\bar u+v\bar v + \f{g}{\phi_0}\phi\bar\phi)\,\text{d}x\text{d}y =  {\int_\Omega \overline U^t S_0 U\,\text{d}x\text{d}y},\hspace{6pt} \norm{U}_H=\{\inner{U}{U}_H\}^{1/2}.
\end{equation*}
Our aim is to prove that $A$ and its adjoint $A^*$ defined below are positive in the sense,
\begin{equation}\label{eq:eq3.1extra}
\begin{cases}
\inner{ AU}{U}_H \geq 0,\hspace{6pt}\forall U\in\mathcal{D}(A),\\
\inner{ A^*U}{U}_H \geq 0,\hspace{6pt}\forall U\in\mathcal{D}(A^*).
\end{cases}
\end{equation}
These properties are needed to apply the Hille-Phillips-Yoshida theorem (see \cite{Yos80,HP74}), see below. The result for $A$ is now easy  thanks to Lemma \ref{lem3.1}. Indeed the following calculations are easy, for $U$ smooth in $\mathcal{D}(A)$:
\begin{equation}\label{eq:eq3.2prime}
\begin{split}
\inner{ AU}{U}_H &=\int_\Omega (u_0u_x + v_0u_y + g\phi_x)u + 
(u_0v_x + v_0v_y + g\phi_y)v \\
&\hspace{20pt}+ \f{g}{\phi_0}(u_0\phi_x + v_0\phi_y + 
\phi_0(u_x+v_y))\phi\, \text{d}x\text{d}y \\
&= (\text{using integration by parts})\\
&= \int_0^{L_2} \big(\f{u_0}{2}(u^2+v^2 +\f{g}{\phi_0}\phi^2) + g\phi 
u\big)\Big{|}_{x=0}^{x=L_1}\text{d}y  \\
&\hspace{20pt}+ \int_0^{L_1} \big(\f{v_0}{2}(u^2+v^2+\f{g}{\phi_0}\phi^2) + 
g\phi v\big)\Big{|}_{y=0}^{y=L_2}\text{d}x \\
&=I_1 + I_2 + I_3 + I_4,
\end{split}
\end{equation}
where $I_1,I_2,I_3,I_4$ respectively stand for the boundary terms at $x=0$, $x=L_1$, $y=0$ and $y=L_2$.
First, the boundary conditions \eqref{eq:eq3.1} imply that $I_1=I_3=0$. We then rewrite $I_2,I_4$ as:
\begin{equation*}
\begin{split}
I_2=\f{u_0}{2}\big[(u+\f{g}{u_0}\phi)^2 + g^2(\f{1}{g\phi_0}-\f{1}{u_0^2})\phi^2 + v^2 \big](L_1,y),\\
I_4=\f{v_0}{2}\big[(v+\f{g}{v_0}\phi)^2 + g^2(\f{1}{g\phi_0}-\f{1}{v_0^2})\phi^2 + u^2 \big](x,L_2),
\end{split}
\end{equation*}
 which are nonnegative under assumption \eqref{eq:eq3.0}.
Therefore, we conclude that $\inner{ AU}{U}_H \geq 0$ for $U$ smooth in $\mathcal{D}(A)$, which is also valid for all $U$ in $\mathcal{D}(A)$ thanks to Lemma \ref{lem3.1}.

We now turn to the definition of the formal adjoint $A^*$ of $A$ and its domain $\mathcal{D}(A^*)$, in the sense of the adjoint of a linear unbounded operator (see \cite{Rud91}). 
For that purpose we first assume that $U\in\mathcal{D}(A)$ and $\overline{U}\in H$ are smooth functions, and then compute
\begin{equation}\label{eq:eq3.2}
\begin{split}
\inner{AU}{\overline{U}}_H &=\int_\Omega (u_0u_x + v_0u_y + g\phi_x)\bar{u} + 
(u_0v_x + v_0v_y + g\phi_y)\bar{v} \\
&\hspace{20pt}+ \f{g}{\phi_0}\big(u_0\phi_x + v_0\phi_y + 
\phi_0(u_x+v_y)\big)\bar{\phi}\, \text{d}x\text{d}y \\
&= J_0 + J_1,
\end{split}
\end{equation}
where $J_0$ stands for the integral on $\Omega$ and $J_1$ for the integral on $\p\Omega$. For $J_0$, we have:
\begin{equation*}
\begin{split}
J_0&=\int_\Omega -(u_0\bar{u}_x + v_0\bar{u}_y + g\bar{\phi}_x)u - 
(u_0\bar{v}_x + v_0\bar{v}_y + g\bar{\phi}_y)v \\
&\hspace{20pt} - \f{g}{\phi_0}\big(u_0\bar{\phi}_x + v_0\bar{\phi}_y + 
\phi_0(\bar{u}_x+\bar{v}_y)\big)\phi\, \text{d}x\text{d}y \\
&=\inner{\mathcal{A}^*\overline{U}}{U}_H,
\end{split}
\end{equation*}
where $\mathcal{A}^*$ is the differential operator defined as follows:
\begin{equation}\label{eq:eq3.3}
\mathcal{A}^*\overline{U}= \begin{pmatrix}
-u_0\bar{u}_x - v_0\bar{u}_y - g\bar{\phi}_x \\
-u_0\bar{v}_x - v_0\bar{v}_y - g\bar{\phi}_y  \\
-u_0\bar{\phi}_x - v_0\bar{\phi}_y - 
\phi_0(\bar{u}_x+\bar{v}_y)  \\
\end{pmatrix}.
\end{equation}
For $J_1$, taking into account the boundary conditions \eqref{eq:eq3.1}, there remains:
\begin{equation*}
\begin{split}
J_1 &= \int_0^{L_2} \big[u_0(u\bar u + v\bar v + \phi\bar\phi) + g(\phi \bar u + u\bar\phi ) ]\big](L_1,y)\text{d}y \\
&\hspace{20pt} + \int_0^{L_1}\big[v_0(u\bar u + v\bar v + \phi\bar\phi) + g(\phi \bar v + v\bar\phi ) \big](x, L_2) \text{d}x.
\end{split}
\end{equation*}
According to \cite{Rud91}, $\mathcal{D}(A^*)$ consists of the $\overline{U}$ in $H$ such that $U\mapsto\inner{AU}{\overline{U}}_H$ is continuous on $\mathcal{D}(A)$ for the topology (norm) of $H$.
If $U$ is restricted to the class of $\mathcal{C}^\infty$ functions with compact support in $\Omega$, then $J_1$ vanishes and $U\mapsto J_0$ can only be continuous if $\mathcal{A}^*\overline{U}$ belongs to $H$.
If $\overline{U}$ and $\mathcal{A}^*\overline{U}$ both belong to $H$, the traces of $\overline{U}$ are well-defined on $\p\Omega$ by observing that Proposition \ref{prop2.1} applies to $\mathcal{A}^*$ as well, and to more general first order linear differential operator with constant coefficients. Hence, the calculations in \eqref{eq:eq3.2} are now valid for any such $\overline{U}$ (and $U$ smooth in $\mathcal{D}(A)$).
We now restrict $U$ to the class of $\mathcal{C}^\infty$ functions on $\overline\Omega$ which belong to $\mathcal{D}(A)$.
Then the expressions above of $J_0$ and $J_1$ show that 
$U\mapsto\inner{AU}{\overline{U}}_H$ can only be continuous in $U$ for the topology (norm) of $H$ if the following boundary conditions are satisfied:
\begin{equation}\label{eq:eq3.4}
\begin{cases}
\bar u=\bar v=\bar\phi = 0, \text{ on } \Gamma_E=\{x =L_1\},  \\
\bar u=\bar v=\bar\phi = 0, \text{ on } \Gamma_N=\{y =L_2\}.
\end{cases}
\end{equation}
{
We now aim to show that
\begin{equation}\label{eq:eq3.10}
\mathcal{D}(A^*) = \{ \overline{U} \in H = L^2(\Omega)^3,\quad \mathcal{A}^*\overline{U}\in H, \text{ and }\overline{U}\text{ satisfies \eqref{eq:eq3.4}}  \},
\end{equation}
and we first conclude from the above calculation that $\mathcal{D}(A^*)$ is included in the space temporarily denoted by $\widetilde{\mathcal{D}}(A^*)$, the right-hand side of \eqref{eq:eq3.10}.
To prove that $\widetilde{\mathcal{D}}(A^*)\subset \mathcal{D}(A^*)$, we only need to observe that \eqref{eq:eq3.2} holds when $U\in\mathcal{D}(A)$ and $\overline U\in\widetilde{\mathcal{D}}(A^*)$ in which case \eqref{eq:eq3.2} reduces to $\inner{AU}{\overline U}_H=\inner{U}{\mathcal{A}\overline U}_H$. This is proved by approximation observing that the smooth functions are dense in $\mathcal{D}(A)$ and in $\widetilde{\mathcal{D}}(A^*)$ respectively. The former density result has already been proven (see Lemma \ref{lem3.1}), and the later one is the object of Lemma \ref{lem3.2} below.
Hence, if $\overline{U}\in \widetilde{\mathcal{D}}(A^*)$, 
then the calculations $\eqref{eq:eq3.2}$ are valid, $J_1$ vanishes, and $U\mapsto\inner{AU}{\overline{U}}_H$ is continuous on $\mathcal{D}(A)$ for the norm of $H$.
Therefore we conclude that $\widetilde{\mathcal{D}}(A^*)\subset \mathcal{D}(A^*)$ and thus \eqref{eq:eq3.10} holds. We then set $A^*\overline U=\mathcal{A}^*\overline U,\,\forall \overline U\in\mathcal{D}(A^*)$.

Let us introduce the density boundary conditions corresponding to \eqref{eq:eq3.4}
\begin{equation}\label{eq:eq3.5}
U \text{ vanishes in a neighborhood of }\Gamma_E\cup\Gamma_N,
\end{equation}
and define the following space of smooth function:
\begin{equation*}
\mathcal{V}^*(\Omega) = \{ \overline U\in \mathcal{C}^\infty(\overline{\Omega})^3,\text{ and }\overline{U}\text{ satisfies \eqref{eq:eq3.5}} \}.
\end{equation*}
Applying Theorem \ref{thm2} as we did for Lemma \ref{lem3.1} with $\vec{\Gamma}=(\Gamma_E\cup\Gamma_N,\Gamma_E\cup\Gamma_N,\Gamma_E\cup\Gamma_N)$, we obtain
\begin{lemma}\label{lem3.2}
$\mathcal{V}^*(\Omega)\cap\widetilde{\mathcal{D}}(A^*)$ is dense in $\widetilde{\mathcal{D}}(A^*)$.
\end{lemma}
Once we know that $\widetilde{\mathcal{D}}(A^*) = \mathcal{D}(A^*)$, Lemma \ref{lem3.2} shows that 
\begin{lemma}\label{lem3.3}
$\mathcal{V}^*(\Omega)\cap\mathcal{D}(A^*)$ is dense in $\mathcal{D}(A^*)$.
\end{lemma}
The proof of the positivity of $A^*$ is similar to the proof for $A$ using Lemma \ref{lem3.3}, we thus omit it here.
}

\subsection{The mixed hyperbolic case}\label{sec-mixed}
{There are two sub-cases in the mixed hyperbolic case, and we only consider one case which corresponds to}
\begin{equation}\label{eq:asp4.1}
u_0^2 < g\phi_0,\hspace{6pt} v_0^2 > g\phi_0,
\end{equation}
{and the other case corresponding to 
\[
u_0^2 > g\phi_0,\hspace{6pt} v_0^2 < g\phi_0
\]
would be similar.

With assumption \eqref{eq:asp4.1}, we see that $a_2$ is negative and $b_1$ is positive and we conclude that
\[
a_1,\,a_3,\, b_1,\,b_2,\,b_3 > 0,\quad\quad a_2 < 0.
\]
As in Subsection \ref{sec-supercritical}, according to the general hyperbolic theory, we specify the boundary conditions for $(\xi,\eta,\zeta)$ at $y=0$ since $b_i$ ($i=1,2,3$) are all positive, and $(\xi,\zeta)$ at $x=0$ since $a_1$ and $a_3$ are positive and $\eta$ at $x=L_1$ since $a_2$ is negative.
Hence, the boundary conditions for $U=(u,v,\phi)$ are}
\begin{equation}\label{eq:eq4.1}
\begin{cases}
\xi = v_0u - u_0v + \kappa_0 \phi = 0, \quad \zeta = u_0u + v_0v + g\phi=0,\text{ on } \Gamma_W=\{x=0\}, \\
\eta =  v_0u - u_0v - \kappa_0 \phi = 0, \text{ on } \Gamma_E=\{x=L_1\}, \\
u = v = \phi = 0, \text{ on } \Gamma_S=\{y=0\},
\end{cases}
\end{equation}
 and we introduce the corresponding density boundary conditions:
\begin{equation}\label{eq:eq4.2}
\begin{cases}
\xi = v_0u - u_0v + \kappa_0 \phi,\;\zeta = u_0u + v_0v + g\phi \text{ vanish in a neighborhood of } \Gamma_W, \\
\eta = v_0u - u_0v - \kappa_0 \phi \text{ vanishes in a neighborhood of } \Gamma_E,\\
U \text{ vanishes in a neighborhood of } \Gamma_S.
\end{cases}
\end{equation}
We then define the function spaces
\begin{equation*}
\begin{split}
&\mathcal{D}(A) = \{ U\in H = L^2(\Omega)^3,\quad \mathcal{A}U\in H,\text{ and }  U\text{ satisfies }\eqref{eq:eq4.1} \}, \\
&\mathcal{V}(\Omega) = \{ U\in \mathcal{C}^\infty(\overline{\Omega})^3, 
\text{ and } U\text{ satisfies }\eqref{eq:eq4.2} \},
\end{split}
\end{equation*}
and write $AU=\mathcal{A}U,\,\forall U\in\mathcal{D}(A)$.
Setting $\vec{\Gamma}=(\Gamma_W\cup\Gamma_S,\Gamma_E\cup\Gamma_S,\Gamma_W\cup\Gamma_S)$, we have that $\hat P^{-1}U|_{\vec{\Gamma}}=0$ is equivalent to \eqref{eq:eq4.1} and $\hat P^{-1}U$ vanishing in a neighborhood of $\vec{\Gamma}$ is the same as \eqref{eq:eq4.2}, where $\hat P$ is as in \eqref{eq:eq2.5}.
Then we obtain the following result from Theorem \ref{thm2} with $Q=\hat P^{-1}, M_1=\mathcal{E}_1, M_2=\mathcal{E}_2, \Lambda=\diag(\lambda_1,\lambda_2,\lambda_3)$.
\begin{lemma}\label{lem4.1}
$\mathcal{V}(\Omega)\cap\mathcal{D}(A)$ is dense in $\mathcal{D}(A)$.
\end{lemma}
\subsubsection{Positivity of $A$ and its adjoint $A^*$}\label{subsec4.1}
As we already calculated in \eqref{eq:eq3.2prime}, we see that $\inner{ AU}{U}_H = I_1 + I_2 + I_3 + I_4$ holds for $U\in \mathcal{C}^\infty(\overline{\Omega})^3$.
If $U$ further satisfies \eqref{eq:eq4.2} (i.e. $U\in\mathcal{V}(\Omega)$), we find $I_3=0$, and we rewrite $I_4=\f{v_0}{2}\big[(v+\f{g}{v_0}\phi)^2 + g^2(\f{1}{g\phi_0}-\f{1}{v_0^2})\phi^2 + u^2 \big](x,L_2)$, which is positive by the assumption $v_0^2>g\phi_0$.
It remains to estimate $I_1,I_2$. {We use the notations $\xi,\eta,\zeta$ and $P$ as in \eqref{eq:eq2.6}}, and we compute $I_2$ as
\begin{equation}
\begin{split}
2I_2 &= [(u,v,\phi)\cdot\begin{pmatrix}
u_0& 0 &g\\
0 &u_0 &0 \\
g &0 &\f{gu_0}{\phi_0}
\end{pmatrix}\cdot\begin{pmatrix}
u \\ v \\ \phi
\end{pmatrix}](L_1, y) \\
&= [(\xi,\eta,\zeta)\cdot {P^t S_0\mc E_1 P}\cdot
\begin{pmatrix}
\xi \\\eta \\\zeta
\end{pmatrix}](L_1,y) \\
&=(\text{using that $\eta=v_0u - u_0v - \kappa_0 \phi$ vanishes in a neighborhood of }\{x=L_1\})\\
&= [(\xi,0,\zeta)\cdot {\diag(a_1,a_2,a_3)}\cdot
\begin{pmatrix}
\xi \\ 0 \\\zeta
\end{pmatrix}](L_1,y) \\
&=[ a_1 \xi^2 + a_3 \zeta^2](L_1,y) \\
&\geq 0.
\end{split}
\end{equation}
Similarly, using that $\xi=v_0u - u_0v + \kappa_0$ and $\zeta=u_0u + v_0v + g\phi$ vanish in a neighborhood of $\{x=0\}$, we compute $2I_1 = - [a_2 \eta^2](0,y) \geq 0$. Therefore, we can conclude that $\inner{ AU}{U}_H \geq 0$ for $U\in\mathcal{V}(\Omega)$, which is also true for all $U\in\mathcal{D}(A)$ thanks to Lemma \ref{lem4.1}.

The formal definition of $A^*$ can be treated similarly as in Subsection \ref{sec-supercritical-sub1}, we thus omit the details. Since we are considering the mixed hyperbolic case, and in order to guarantee that $U\mapsto\inner{AU}{\overline{U}}_H$ is continuous on $\mathcal{D}(A)$ (see Subsec. \ref{sec-supercritical-sub1}), the following boundary conditions need to be satisfied:
\begin{equation}\label{eq:eq4.4}
\begin{cases}
v_0\bar u - u_0\bar v - \kappa_0 \bar\phi = 0, \text{ on } \Gamma_W=\{x=0\}, \\
v_0\bar u - u_0\bar v + \kappa_0 \bar\phi = u_0\bar u + v_0\bar v + g\bar\phi=0,\text{ on } \Gamma_E=\{x=L_1\}, \\
\bar u = \bar v = \bar\phi = 0, \text{ on } \Gamma_N=\{y=L_2\}.
\end{cases}
\end{equation}
{Arguing exactly as in Subsection \ref{sec-supercritical-sub1}}, we conclude that
\begin{equation*}
\mathcal{D}(A^*) = \{ \overline{U} \in H = L^2(\Omega)^3,\quad A^*\overline{U}\in H, \text{ and }\overline{U}\text{ satisfies \eqref{eq:eq4.4}}  \},
\end{equation*}
and write $A^*\overline U=\mathcal{A}^*\overline U,\,\forall \overline U\in\mathcal{D}(A^*)$.

We also introduce the corresponding density boundary conditions:
\begin{equation}\label{eq:eq4.5}
\begin{cases}
v_0\bar u - u_0\bar v - \kappa_0 \bar \phi \text{ vanishes in a neighborhood of } \Gamma_W,\\
v_0\bar u - u_0\bar v + \kappa_0 \bar \phi,u_0\bar u + v_0\bar v + g\bar \phi \text{ vanish in a neighborhood of } \Gamma_E, \\
\overline U \text{ vanishes in a neighborhood of } \Gamma_N,
\end{cases}
\end{equation}
and define the corresponding space of smooth function:
\begin{equation*}
\mathcal{V}^*(\Omega) = \{ \overline{U} \in \mathcal{C}(\overline\Omega)^3, \text{ and }\overline{U}\text{ satisfies \eqref{eq:eq4.5}}  \}.
\end{equation*}
Similarly as in Subsection \ref{sec-supercritical-sub1}, we obtain
\begin{lemma}
$\mathcal{V}^*(\Omega)\cap\mathcal{D}(A^*)$ is dense in $\mathcal{D}(A^*)$, and $A^*$ is positive in the sense of \eqref{eq:eq3.1extra}.
\end{lemma}

\subsection{The fully hyperbolic subcritical case}\label{sec-fully-hyperbolic}
{The fully hyperbolic subcritical case corresponds to the case where:}
\begin{equation}\label{eq:asp5.1}
u_0^2 < g\phi_0,\hspace{6pt} v_0^2 < g\phi_0, \hspace{6pt} u_0^2+v_0^2>g\phi_0.
\end{equation}

{Under assumption \eqref{eq:asp5.1}, we see that $a_2$ and $b_1$ are both negative and we conclude that
\[
a_1,\,a_3,\,b_2,\,b_3 > 0,\quad\quad a_2,\,b_1 < 0.
\]
As in Subsection \ref{sec-supercritical}, according to the general hyperbolic theory, we specify the boundary conditions for $(\eta,\zeta)$ at $y=0$ and $\xi$ at $y=L_2$, and $(\xi,\zeta)$ at $x=0$ and $\eta$ at $x=L_1$.
Hence, the boundary conditions for $U=(u,v,\phi)$ are}
\begin{equation}\label{eq:eq5.1}
\begin{cases}
\xi=v_0u - u_0v + \kappa_0 \phi =0,\quad \zeta= u_0u + v_0v + g\phi = 0, \text{ on } \Gamma_W=\{x=0\}, \\
\eta=v_0u - u_0v - \kappa_0 \phi = 0, \text{ on } \Gamma_E=\{x=L_1\}, \\
\eta=v_0u - u_0v - \kappa_0 \phi =0,\quad \zeta = u_0u + v_0v + g\phi = 0, \text{ on } \Gamma_S=\{y=0\}, \\
\xi=v_0u - u_0v + \kappa_0 \phi = 0, \text{ on } \Gamma_N=\{y=L_2\},
\end{cases}
\end{equation}
 and we introduce the corresponding density boundary conditions:
\begin{equation}\label{eq:eq5.2}
\begin{cases}
\xi=v_0u - u_0v + \kappa_0 \phi,\,\zeta=u_0u + v_0v + g\phi \text{ vanish in a neighborhood of } \Gamma_W, \\
\eta=v_0u - u_0v - \kappa_0 \phi \text{ vanishes in a neighborhood of } \Gamma_E, \\
\eta=v_0u - u_0v - \kappa_0 \phi,\,\zeta=u_0u + v_0v + g\phi \text{ vanish in a neighborhood of } \Gamma_S, \\
\xi=v_0u - u_0v + \kappa_0 \phi \text{ vanishes in a neighborhood of } \Gamma_N.
\end{cases}
\end{equation}
We then define the function spaces
\begin{equation*}
\begin{split}
&\mathcal{D}(A) = \{ U\in H = L^2(\Omega)^3,\quad \mathcal{A}U\in H,\text{ and }  U\text{ satisfies }\eqref{eq:eq5.1} \},\\
&\mathcal{V}(\Omega) = \{ U\in \mathcal{C}^\infty(\overline{\Omega})^3, 
\text{ and } U\text{ satisfies }\eqref{eq:eq5.2} \},
\end{split}
\end{equation*}
and write $AU=\mathcal{A}U,\,\forall U\in\mathcal{D}(A)$.
Setting $\vec{\Gamma}=(\Gamma_W\cup\Gamma_N,\Gamma_E\cup\Gamma_S,\Gamma_W\cup\Gamma_S)$, we have that $\hat P^{-1}U|_{\vec{\Gamma}}=0$ is equivalent to \eqref{eq:eq5.1} and $\hat P^{-1}U$ vanishing in a neighborhood of $\vec{\Gamma}$ is the same as \eqref{eq:eq5.2}, where $P$ is as in \eqref{eq:eq2.5}.
Then we obtain the following result from Theorem \ref{thm2} with $Q=\hat P^{-1}, M_1=\mathcal{E}_1, M_2=\mathcal{E}_2, \Lambda=\diag(\lambda_1,\lambda_2,\lambda_3)$.
\begin{lemma}\label{lem5.1}
$\mathcal{V}(\Omega)\cap\mathcal{D}(A)$ is dense in $\mathcal{D}(A)$.
\end{lemma}

\subsubsection{Positivity of $A$ and its adjoint $A^*$}
As indicated in Subsection \ref{subsec4.1}, we have that $\inner{ AU}{U}_H = I_1 + I_2 + I_3 + I_4$ holds for $U\in \mathcal{C}^\infty(\overline{\Omega})^3$.
Suppose that $U$ belongs to $\mathcal{V}(\Omega)$. In order to estimate $I_1,I_2,I_3,I_4$, 
we still use the notations {$\xi,\eta,\zeta$ and $P$  as in \eqref{eq:eq2.6}},
and $I_1,I_2$ are both nonnegative since the estimates for $I_1,I_2$ are exactly the same as in Subsection \ref{subsec4.1}. Proceeding exactly as for $I_2$, we compute $I_4$ as:
\begin{equation}
\begin{split}
2I_4 &= [(u,v,\phi)\cdot\begin{pmatrix}
v_0& 0 &0\\
0 &v_0 &g \\
0 &g &\f{gv_0}{\phi_0}
\end{pmatrix}\cdot\begin{pmatrix}
u \\ v \\ \phi
\end{pmatrix}](x, L_2) \\
&= [(\xi,\eta,\zeta)\cdot {P^t S_0\mc E_2 P}
\begin{pmatrix}
\xi \\ \eta \\ \zeta
\end{pmatrix}](x,L_2) \\
&=(\text{using that $\xi=v_0u - u_0v + \kappa_0 \phi$ vanishes in a neighborhood of }\{y=L_2\})\\
&= [(0,\eta,\zeta)\cdot {\diag(b_1,b_2,b_3)}\cdot
\begin{pmatrix}
0 \\ \eta \\ \zeta
\end{pmatrix}](x,L_2) \\
&=[ b_2\eta^2 + b_3\zeta^2](x,L_2) \\
&\geq 0.
\end{split}
\end{equation}
Similarly, using that $\eta=v_0u - u_0v - \kappa_0$ and $\zeta=u_0u + v_0v + g\phi$ vanish in a neighborhood of $\{y=0\}$, we compute $2I_3 = - [b_1 \xi^2](x,0) \geq 0$. Therefore, we conclude that $\inner{ AU}{U}_H \geq 0$ for $U\in \mathcal{V}(\Omega)$; this implies that $\inner{ AU}{U}_H \geq 0$ also holds for all $U\in\mathcal{D}(A)$ by virtue of Lemma \ref{lem5.1}. Hence $A$ is positive.

Taking similar arguments in Subsection \ref{sec-supercritical-sub1} and Subsection \ref{subsec4.1}, we will obtain the same results for $A^*$ the adjoint operator of $A$. We thus only state them below without the proof.

In the fully hyperbolic subcritical case,
we first introduce the following boundary conditions ,
\begin{equation}\label{eq:eq5.4}
\begin{cases}
v_0\bar u - u_0\bar v - \kappa_0 \bar\phi = 0, \text{ on }\Gamma_W=\{x=0\}, \\
v_0\bar u - u_0\bar v + \kappa_0 \bar\phi = u_0\bar u + v_0\bar v + g\bar\phi=0,\text{ on }\Gamma_E=\{x=L_1\}, \\
v_0\bar u - u_0\bar v + \kappa_0 \bar\phi = 0, \text{ on } \Gamma_S=\{y=0\}, \\
v_0\bar u - u_0\bar v - \kappa_0 \bar\phi = u_0\bar u + v_0\bar v + g\bar\phi=0,\text{ on } \Gamma_N=\{y=L_2\};
\end{cases}
\end{equation}
and the corresponding density boundary conditions:
\begin{equation}\label{eq:eq5.5}
\begin{cases}
v_0\bar u - u_0\bar v - \kappa_0 \bar \phi \text{ vanishes in a neighborhood of } \Gamma_W,\\
v_0\bar u - u_0\bar v + \kappa_0 \bar \phi,u_0\bar u + v_0\bar v + g\bar \phi \text{ vanish in a neighborhood of } \Gamma_E, \\
v_0\bar u - u_0\bar v + \kappa_0 \bar \phi \text{ vanishes in a neighborhood of }\Gamma_S,\\
v_0\bar u - u_0\bar v - \kappa_0 \bar \phi,u_0\bar u + v_0\bar v + g\bar \phi \text{ vanish in a neighborhood of } \Gamma_N.
\end{cases}
\end{equation}
We then define the function spaces:
\begin{equation*}
\begin{split}
&\mathcal{D}(A^*) = \{ \overline{U} \in H = L^2(\Omega)^3,\quad A^*\overline{U}\in H, \text{ and }\overline{U}\text{ satisfies \eqref{eq:eq5.4}}  \},\\
&\mathcal{V}^*(\Omega) = \{ \overline{U} \in \mathcal{C}(\overline\Omega)^3, \text{ and }\overline{U}\text{ satisfies \eqref{eq:eq5.5}}  \},
\end{split}
\end{equation*}
and write $A^*\overline U=\mathcal{A}^*\overline U,\,\forall \overline U\in\mathcal{D}(A^*)$.
Then we have the following results.
\begin{lemma}
$\mathcal{V}^*(\Omega)\cap\mathcal{D}(A^*)$ is dense in $\mathcal{D}(A^*)$, and $A^*$ is positive in the sense of \eqref{eq:eq3.1extra}.
\end{lemma}

\section{The elliptic-hyperbolic case}\label{sec-elliptic-hyperbolic}

We aim in this section to study the elliptic-hyperbolic case (we also call it the mixed subcritical case) which corresponds to
\begin{equation}\label{eq:asp6.0}
u_0^2+v_0^2 < g\phi_0, \hspace{5pt}\text{implying}\hspace{5pt}u_0^2<g\phi_0,\hspace{3pt} v_0^2<g\phi_0, 
\end{equation}
and which embodies some elliptic operators.
\subsection{A preliminary regularity theorem}\label{sec-regularity}

In this subsection, we introduce some necessary preliminary results for { the elliptic operator embodied in the stationary 2D shallow water equations operator,}
namely we prove some theorems regarding functions defined on the domain $\Omega=(0,L_1)\times(0,L_2)$. {These theorems will be very useful in verifying the positivity of the elliptic mode in the shallow water equations operator in Subsection \ref{sec-mixed-subcritical}}.
Using the notations in Subsection \ref{sec-density}, we introduce the following boundary conditions:
\begin{equation}\label{eq:eq6.1}
\begin{cases}
 \theta_1 = 0 \text{ on } \Gamma, \\
 \theta_2 = 0 \text{ on } \Gamma^c,
 \end{cases}
\end{equation}
where $\Gamma^c$ is the complement of $\Gamma$ with respect to the boundary $\p\Omega$. Here we choose $\Gamma=\Gamma_W\cup\Gamma_S$, which is the case considered in Subsection \ref{sec-mixed-subcritical}. We define
\begin{equation*}
V=\{ \Theta=(\theta_1,\theta_2)^t\in H^1(\Omega)^2 \ |\ \Theta\text{ satisfies \eqref{eq:eq6.1} } \}.
\end{equation*}
Note that $\norm{\nabla \Theta}_{L^2(\Omega)^2}$ is a natural norm on $V$ thanks to the Poincar\'e inequality.

{
We assume that $\alpha_1,\alpha_2,\beta_1,\beta_2$ are four real constants satisfying
 \begin{equation}\label{eq:con6.1}
\alpha_1,\,\alpha_2>0,\quad\quad \alpha_2\beta_1-\alpha_1\beta_2 \neq 0,
 \end{equation}
which is the case considered in Subsection \ref{sec-mixed-subcritical}.
We then define the unbounded operator $\mathcal{T}$ on $L^2(\Omega)^2$ with
\begin{equation}\label{eq:eq6.3}
\mathcal{T}\Theta=T_1\Theta_x+T_2\Theta_y,\quad\forall\,\Theta\in \mc D(\mc T),
\end{equation}
where
\begin{equation*}
T_1=\begin{pmatrix}
\alpha_1 & \beta_1\\
\beta_1 & -\alpha_1
\end{pmatrix},\hspace{6pt}
T_2=\begin{pmatrix}
\alpha_2 & \beta_2\\
\beta_2 & -\alpha_2
\end{pmatrix},
\end{equation*}
and
\[
\mc D(\mc T)=\set{ \Theta=(\theta_1,\theta_2)^t\in L^2(\Omega)^2 \ |\ \mathcal{T}\Theta\in L^2(\Omega)^2,\,\Theta\text{ satisfies \eqref{eq:eq6.1} } }.
\]
\begin{rmk}\label{rmk6.0}
We remark here that the choice of the boundary conditions in \eqref{eq:eq6.1} is based on the signs of $\alpha_1,\alpha_2$ in the assumption \eqref{eq:con6.1}, and we could also choose other boundary conditions depending on the signs of $\beta_1,\beta_2$ and the result would be similar. Indeed, we could consider a change of variables as follows:
\begin{equation}
\overline\Theta=\begin{pmatrix}
\bar \theta_1\\
\bar \theta_2
\end{pmatrix}=Q\Theta:=\f{1}{\sqrt{2}}\begin{pmatrix}
1&-1\\
1&1
\end{pmatrix}\begin{pmatrix}
\theta_1\\
\theta_2
\end{pmatrix}=\f{1}{\sqrt{2}}\begin{pmatrix}
\theta_1-\theta_2\\
\theta_1+\theta_2
\end{pmatrix}.
\end{equation}
Then we define a new operator $\overline{\mc T}$ for the new variable $\overline\Theta=(\bar \theta_1,\bar\theta_2)$:
\begin{equation}
 \overline{\mc T}\,\overline\Theta := Q\mc TQ^{-1}\overline\Theta = \overline T_1\overline{\Theta}_x + \overline T_2\overline{\Theta}_y ,
\end{equation}
where
\begin{equation*}
\overline T_1=QT_1Q^{-1}=\begin{pmatrix}
-\beta_1& \alpha_1\\
\alpha_1&\beta_1 
\end{pmatrix},\hspace{6pt}
\overline T_2=QT_2Q^{-1}=\begin{pmatrix}
-\beta_2&\alpha_2\\
\alpha_2&\beta_2
\end{pmatrix}.
\end{equation*}
In order to be consistent with the case that we consider in Sections \ref{sec-mixed-subcritical}, we assume that $\beta_1>0$ and $\beta_2<0$, and we can choose the following boundary conditions
\begin{equation}
\begin{cases}
 \bar\theta_1 = 0\; \text{ on }\; \Gamma_E\cup\Gamma_S,\\
 \bar\theta_2 = 0\; \text{ on }\; \Gamma_W\cup\Gamma_N. 
 \end{cases}
\end{equation}
\qed
\end{rmk}

\subsubsection{Properties of the operator $\mathcal{T}$}
Notice that $\mathcal{T}$ is symmetric and elliptic, and actually we have the following result.
\begin{thm}\label{thm6.1}
We assume that \eqref{eq:con6.1} holds. Then the domain $\mc D(\mc T)$ of $\mc T$ is the space $V$.
\end{thm}
}

\begin{proof}
{It is obvious that $V$ in included in $\mc D(\mc T)$, we only need to prove that $\mc D(\mc T)\subset V$.}

Let us recall a basic fact from linear algebra. We endow $\mathbb{R}^2$ with its usual dot product and induced norm:
\begin{equation*}
x\cdot y=x_1y_1+x_2y_2,\hspace{6pt}|x|_2=\sqrt{x_1^2+x_2^2},
\end{equation*} 
where $x=(x_1,x_2),y=(y_1,y_2)\in\mathbb{R}^2$. Let $\mathbb{T}$ be a linear transformation from $\mathbb{R}^2$ to $\mathbb{R}^2$ with $T_0$ as its matrix representation, where
\begin{equation*}
T_0=\begin{pmatrix}
\alpha_1&\alpha_2\\
\beta_1&\beta_2
\end{pmatrix}.
\end{equation*}
Condition \eqref{eq:con6.1} shows that $T_0$ is non-singular, which implies that $\mathbb{T}$ is an isomorphism. Hence we have that
\begin{equation}\label{eq:eq6.4}
\f{1}{c_1}|x|_2 \leq |\mathbb{T}x|_2 \leq c_2 |x|_2,
\end{equation}
where $c_2$ (resp. $c_1$) is the spectral norm of $T_0$ (resp. $T_0^{-1}$), i.e. the square root of the largest eigenvalue of the positive-definite matrix $T_0^tT_0$ (resp. $T_0^{-t}T_0^{-1}$).

{We now formally derive \emph{a priori estimates}, and thus compute}
\begin{equation}\label{eq:eq6.5p}
\begin{split}
\norm{\mathcal{T}\Theta}_{L^2(\Omega)^2}^2 
&=\int_{\Omega} (\Theta_x^t T_1^t+\Theta_y^t T_2^t)(T_1\Theta_x+T_2\Theta_y)\, \text{d}x\text{d}y \\
&=\int_{\Omega} (\alpha_1^2+\beta_1^2)(\theta_{1x}^2+\theta_{2x}^2) + (\alpha_2^2+\beta_2^2)(\theta_{1y}^2+\theta_{2y}^2) \\
&\hspace{25pt}+2(\alpha_1\alpha_2+\beta_1\beta_2)(\theta_{1x}\theta_{1y}+\theta_{2x}\theta_{2y}) \\
&\hspace{25pt}+2(\alpha_2\beta_1-\alpha_1\beta_2)(\theta_{2x}\theta_{1y} - \theta_{1x}\theta_{2y})\,\text{d}x\text{d}y. \\
\end{split}
\end{equation}

To dispense with the last term in the integral of \eqref{eq:eq6.5p}, 
we use a result from \cite{Gri85} (see Lemma 4.3.1.3), which implies:
\begin{lemma}\label{lem6.1}
The identity
\begin{equation*}
\int_\Omega \theta_{2x}\theta_{1y}\, \text{d}x\text{d}y = \int_\Omega \theta_{1x}\theta_{2y}\,\text{d}x\text{d}y
\end{equation*}
holds for all $\Theta=(\theta_1,\theta_2)^t\in H^1(\Omega)^2$ satisfying \eqref{eq:eq6.1} (i.e. $\Theta\in V$).
\end{lemma}

Thanks to Lemma \ref{lem6.1}, \eqref{eq:eq6.5p} gives
\begin{equation}\label{eq:eq6.5}
\begin{split}
\norm{\mathcal{T}\Theta}_{L^2(\Omega)^2}^2 &=\int_{\Omega} (\alpha_1\theta_{1x} + \alpha_2\theta_{1y})^2 + (\beta_1\theta_{1x}+\beta_2\theta_{1y})^2 \\
&\hspace{25pt}+(\alpha_1\theta_{2x} + \alpha_2\theta_{2y})^2 + (\beta_1\theta_{2x}+\beta_2\theta_{2y})^2\, \text{d}x\text{d}y\\
&=\int_{\Omega} |\mathbb{T}\nabla\theta_1|_2^2 + |\mathbb{T}\nabla\theta_2|_2^2\, \text{d}x\text{d}y.
\end{split}
\end{equation}
With the help of \eqref{eq:eq6.4}, \eqref{eq:eq6.5} immediately implies that
\begin{equation}\label{eq:eq6.6}
\f{1}{c_1}\norm{\nabla\Theta}_{L^2(\Omega)^2}\leq \norm{\mathcal{T}\Theta}_{L^2(\Omega)^2}\leq c_2 \norm{\nabla\Theta}_{L^2(\Omega)^2}.
\end{equation}

{
Now for any $\Theta\in \mc D(\mc T)$, we set $F=\mc T\Theta\in L^2(\Omega)^2$ and approximate $F$ by smooth functions $F_n\in \mc D(\Omega)^2$ ($n=1,2,3,\cdots$). Proposition \ref{prop6.2} below shows that there exists $\Theta_n\in V$ such that
$\mc T\Theta_n = F_n$ for all $n=1,2,3,\cdots$. The a priori estimates \eqref{eq:eq6.6} and the Poincar\'e inequality show that $\Theta_n$ is uniformly bounded in $H^1(\Omega)^2$, which shows that $\Theta_n$ (up to a subsequence) converges weakly to some $\overline\Theta\in H^1(\Omega)^2$. Hence, we obtain that
\[
\mc T\Theta = F = \lim_{n\rightarrow \infty}\mc T \Theta_n  = \mc T\overline\Theta
\]
in the sense of distributions, and hence $\mc T (\Theta-\overline\Theta)=0$. The uniqueness in Proposition \ref{prop6.2} below implies that $\Theta=\overline\Theta$. Therefore, $\Theta\in V$, and we thus completed the proof of Theorem~\ref{thm6.1}. 
}
\end{proof}

\begin{prop}\label{prop6.2}
For every given $\Psi=(\psi_1,\psi_2)^t\in \mathcal{C}_c^1(\Omega)^2$, the problem $\mathcal{T}\Theta = \Psi$ has a unique solution $\Theta\in V$. {The solution of $\mathcal{T}\Theta = \Psi$ is also unique in $\mc D(\mc T)$.}
\end{prop}
\begin{proof}
Without loss of generality, we first assume that,
\begin{equation}\tag{$\ref{eq:con6.1}'$}\label{eq:con6.1p}
\alpha_2\beta_1 - \alpha_1\beta_2 =1.
\end{equation}
In order to make the operator $\mc T$ simpler, we introduce a new coordinate system $(x',y')$ such that
\begin{equation*}
\begin{pmatrix}
x' \\ y'
\end{pmatrix}=\begin{pmatrix}
\beta_2 & -\beta_1 \\
\alpha_2 & -\alpha_1
\end{pmatrix}\begin{pmatrix}
x \\ y
\end{pmatrix}=\begin{pmatrix}
\beta_2x-\beta_1y\\
\alpha_2x-\alpha_1y
\end{pmatrix},
\end{equation*}
which is equivalent to
\begin{equation*}
\begin{pmatrix}
x \\ y
\end{pmatrix}=\begin{pmatrix}
-\alpha_1 & \beta_1 \\
-\alpha_2 & \beta_2
\end{pmatrix}\begin{pmatrix}
x' \\ y'
\end{pmatrix}=\begin{pmatrix}
-\alpha_1x' + \beta_1y'\\
-\alpha_2x'+\beta_2y'
\end{pmatrix}.
\end{equation*}
We denote by $\Gamma_i'$ the image of $\Gamma_i$ by this transformation for all $i\in\{W,E,S,N\}$, and denote by $\Omega', \Gamma',\theta',\Psi'$ and the gradient $\nabla'$ the transforms of $\Omega,\Gamma,\theta,\Psi$ and the gradient $\nabla$ respectively.
Now, direct computation gives
\begin{equation}\label{eq:eq6.11}
\nabla\theta =\begin{pmatrix}
\beta_2 & \alpha_2 \\
-\beta_1 & -\alpha_1
\end{pmatrix}\nabla'\theta'.
\end{equation}
In the new coordinate system $(x',y')$, the boundary conditions \eqref{eq:eq6.1} read
\begin{equation}\tag{$\ref{eq:eq6.1}'$}\label{eq:eq6.1p}
\begin{cases}
\theta_1' = 0 \text{ on } \Gamma', \\
\theta_2' = 0 \text{ on } \Gamma'^c,
\end{cases}
\end{equation}
where $\Gamma'=\Gamma_1'\cup\Gamma_3'$. We also denote by $V'$ the function space corresponding to $V$:
\begin{equation*}
V'=\{ \Theta'=(\theta_1',\theta_2')^t\in H^1(\Omega')^2 \ |\ \Theta'=(\theta_1',\theta_2')^t\text{ satisfies \eqref{eq:eq6.1p} } \}.
\end{equation*}
In the new coordinate system $(x',y')$, the operator $\mathcal{T}$ reads
\begin{equation}
\begin{split}
\mathcal{T}\Theta' &= \begin{pmatrix}
\alpha_1 & \alpha_2 \\
\beta_1 & \beta_2
\end{pmatrix}\nabla\theta_1 +\begin{pmatrix}
\beta_1 & \beta_2 \\
-\alpha_1 & -\alpha_2
\end{pmatrix}\nabla\theta_2 
\\
&=(\text{using \eqref{eq:eq6.11}}) \\
&=\begin{pmatrix}
\alpha_1 & \alpha_2 \\
\beta_1 & \beta_2
\end{pmatrix}\begin{pmatrix}
\beta_2 & \alpha_2 \\
-\beta_1 & -\alpha_1
\end{pmatrix}\nabla'\theta_1' +\begin{pmatrix}
\beta_1 & \beta_2 \\
-\alpha_1 & -\alpha_2
\end{pmatrix}\begin{pmatrix}
\beta_2 & \alpha_2 \\
-\beta_1 & -\alpha_1
\end{pmatrix}\nabla'\theta_2' 
\\
&=(\text{using assumption \eqref{eq:con6.1p}})\\
&=\begin{pmatrix}
-1 & 0 \\
0 & 1
\end{pmatrix}\nabla'\theta_1' +\begin{pmatrix}
0 & 1 \\
1 & 0
\end{pmatrix}\nabla'\theta_2' 
\\
&=\begin{pmatrix}
-\theta_{1x'}'  + \theta_{2y'}' \\
\theta_{1y'}' + \theta_{2x'}'
\end{pmatrix},
\end{split}
\end{equation}
and of course, we have $\Psi'\in\mc C_c^1(\Omega')$. We are seeking $\Theta'\in V'$ satisfying $\mathcal{T}\Theta' = \Psi'$, that is
\begin{equation}\label{eq:eq6.13}
\begin{cases}
-\theta_{1x'}' + \theta_{2y'}' = \psi_1', \\
\theta_{1y'}' + \theta_{2x'}' = \psi_2'.
\end{cases}
\end{equation}
Differentiating $\eqref{eq:eq6.13}_1$ with respect to $x'$ and $\eqref{eq:eq6.13}_2$ with respect to $y'$, and subtracting these two equations, we find the elliptic equation
\begin{equation}\label{eq:eq6.14}
\Delta'\theta_1' = -\psi_{1x'}' + \psi_{2y'}',
\end{equation}
where $\Delta'$ denotes the Laplace operator in the new coordinate system $(x',y')$. We associate to equation \eqref{eq:eq6.14} the boundary conditions
\begin{equation}\label{eq:eq6.15}
 \theta_1'=0 \text{ on } \Gamma',
 \end{equation}
which is already contained in \eqref{eq:eq6.1p}. For the boundary $\Gamma'^c$, suitable boundary conditions can be obtained as follows. We first denote by $\nu_i'$ the unit normal vector to $\Gamma_i'$, and $\tau_i'$ the unit tangent vector on $\Gamma_i'$, for all $i\in\{W,E,S,N\}$. On $\Gamma_E'=\{-\alpha_1x'+\beta_1y'=L_1\}$, we have $\p\theta_2'/\p\tau_E'=0$ since $\theta_2'=0$ on $\Gamma_E'$. Noticing that $(\beta_1,\alpha_1)$ is parallel to $\tau_E'$, we find
\begin{equation}\label{eq:eq6.16}
0=(\beta_1,\alpha_1)\begin{pmatrix}
\theta_{2x'}' \\ \theta_{2y'}'
\end{pmatrix} = \beta_1\theta_{2x'}' + \alpha_1\theta_{2y'}',\text{ on } \Gamma_E'.
\end{equation}
Multiplying $\eqref{eq:eq6.13}_1$ by $\alpha_1$, $\eqref{eq:eq6.13}$ by $\beta_1$, and adding these two equations, we find
\begin{equation*}
-\alpha_1\theta_{1x'}' + \beta_1\theta_{1y'}' + \alpha_1\theta_{2y'}'+\beta_1\theta_{2x'}' = \alpha_1\psi_1'+\beta_1\psi_2',
\end{equation*}
which, by using \eqref{eq:eq6.16} and noticing that $\Psi'=(\psi_1',\psi_2')$ vanishes on $\Gamma_2'$, implies that
\begin{equation}\label{eq:eq6.17}
(-\alpha_1,\beta_1)\begin{pmatrix}
\theta_{1x'}' \\ \theta_{1y'}'
\end{pmatrix}=-\alpha_1\theta_{1x'}' + \beta_1\theta_{1y'}' = 0, \text{ on } \Gamma_2'.
\end{equation}
Observing that $\nu_E'$ is parallel to $(-\alpha_1,\beta_1)$, we then associate to \eqref{eq:eq6.14} the following boundary condition on $\Gamma_E'$:
\begin{equation}\label{eq:eq6.18}
\f{\p\theta_1'}{\p\nu_E'} = 0, \text{ on } \Gamma_E',
\end{equation}
which is equivalent to \eqref{eq:eq6.17}.
On $\Gamma_N'=\{-\alpha_2x'+\beta_2y' = L_2\}$, we have $\p\theta_2'/\p\tau_N'=0$ since $\theta_2'=0$ on $\Gamma_N'$. Noticing that $(\beta_2,\alpha_2)$ is parallel to $\tau_N'$, and $\nu_N'$ is parallel to $(-\alpha_2,\beta_2)$, similar computations show that
we need to associate to \eqref{eq:eq6.14} the following boundary condition on $\Gamma_N'$:
\begin{equation}\label{eq:eq6.19}
\f{\p\theta_1'}{\p\nu_N'} = 0, \text{ on } \Gamma_N'.
\end{equation}

The existence and uniqueness of a solution $\theta_1'\in H^1(\Omega')$ of \eqref{eq:eq6.14}-\eqref{eq:eq6.15} and \eqref{eq:eq6.18}-\eqref{eq:eq6.19} follows from Lemma 4.4.3.1 of \cite{Gri85} with $\Omega=\Omega', \mathscr{D} = \Gamma', \mathscr{N} = \Gamma'^c, \beta_j=0, f = -\psi_{1x'}+\psi_{2y'}$.

Following the arguments for $\theta_1'$, we find the following equation and boundary conditions for $\theta_2$:
\begin{equation}\label{eq:eq6.20}
\begin{cases}
\Delta'\theta_2' = \psi_{1y'}' + \psi_{2x'}', \\
\theta_2' = 0, \text{ on } \Gamma'^c=\Gamma_E'\cup\Gamma_N', \\
\f{\p\theta_2'}{\p\nu_W'} = 0,\text{ on } \Gamma_W', \\
\f{\p\theta_2'}{\p\nu_S'} = 0,\text{ on } \Gamma_S'.
\end{cases}
\end{equation}
We thus also have a unique solution $\theta_2'\in H^1(\Omega')$ of \eqref{eq:eq6.20} thanks to Lemma 4.4.3.1 of \cite{Gri85} again. 

In conclusion, in the new coordinate system $(x',y')$, we find a unique solution $\Theta'$ that belongs to $V'$, and solves the problem $\mathcal{T}\Theta' = \Psi'$. Transforming back to the original coordinate system $(x,y)$, we obtain $\Theta\in V$ satisfying $\mathcal{T}\Theta = \Psi$. Hence, the proof of Proposition \ref{prop6.2}  is complete.

{
The uniqueness in $\mc D(\mc T)$ amounts now to showing that $\mc T$ is one to one on $\mc D(\mc T)$, that is $\mc T\Theta = 0$ implies $\Theta=0$. This result comes from the fact that its  transformed function $\Theta'=(\theta_1',\theta_2')$ is the solution of \eqref{eq:eq6.13} and the similar problem for $\theta_1'$ (\eqref{eq:eq6.14}-\eqref{eq:eq6.15},\eqref{eq:eq6.18}-\eqref{eq:eq6.19}). Hence $\theta_1'=\theta_2'=0$, and $\Theta=0$. The proof of Proposition \ref{prop6.2} is now complete.
}
\end{proof}

{
\begin{thm}\label{thm6.2}
We assume that \eqref{eq:con6.1} holds. Then the operator $\mc T$ is positive, i.e. $\inner{\mc T \Theta}{\Theta}\geq 0$ holds for any $\Theta\in \mc D(\mc T)=V$.
\end{thm}
\begin{proof}
For $\Theta\in V$, integrations by parts, which are valid since $\Theta\in H^1(\Omega)$, yield
\begin{equation*}
\begin{split}
\inner{\mc T \Theta}{\Theta}&=\int_\Omega  \Theta^t T_1\Theta_x +  \Theta^t T_2\Theta_y\, \text{d}x\text{d}y\\
&= \f12\int_0^{L_2}  \Theta^t T_1\Theta\Big|_{x=0}^{x=L_1} \text{d}y + \f12\int_0^{L_1}  \Theta^t T_2\Theta\Big|_{y=0}^{y=L_2} \text{d}x\\
&\geq 0,
\end{split}
\end{equation*}
where we used  the boundary conditions \eqref{eq:eq6.1} and the assumption $\alpha_1,\alpha_2>0$.
\end{proof}
}

\begin{thm}\label{thm6.3}
We assume that \eqref{eq:con6.1} holds. For any $F=(f_1,f_2)^t\in L^2(\Omega)^2$, the problem $\mathcal{T}\Theta = F$ has a unique solution $\Theta\in V$.
\end{thm}
\begin{proof}
Let $l(\overline{\Theta})=\inner{F}{\mathcal{T}\overline{\Theta}}$ and $a(\Theta,\overline{\Theta})=\inner{\mathcal{T}\Theta}{\mathcal{T}\overline{\Theta}}_{L^2(\Omega)^2}$;
then it is not hard to check that the linear form $l$ is continuous on $V$, and the bilinear form $a$ is continuous on $V\times V$. Observing that the form $a$ is coercive on $V$ thanks to the estimate \eqref{eq:eq6.6}, we obtain that there exists a unique $\Theta\in V$ such that 
\begin{equation}\label{eq:eq6.91}
a(\Theta,\overline{\Theta})=l(\overline{\Theta}),
\end{equation}
for all $\overline{\Theta}\in V$ thanks to the Lax-Milgram Theorem. 

We now interpret \eqref{eq:eq6.91} in the distribution sense.
For any $\Psi\in\mathcal{D}(\Omega)^2$, we find $\overline{\Theta}\in V$ such that $\mathcal{T}\overline{\Theta}=\Psi$ by virtue of Proposition \ref{prop6.2}, and then \eqref{eq:eq6.91} shows that $\mathcal{T}\Theta =F$ in the sense of distributions, and hence in $L^2(\Omega)^2$ since $\mathcal{D}(\Omega)^2$ is dense in $L^2(\Omega)^2$.
\end{proof}

\subsubsection{Properties of the adjoint operator $\mathcal{T}^*$}
We now turn to the adjoint $\mc T^*$ of $\mc T$. We will develop a similar result for $\mathcal{T}^*$. 
The adjoint $\mathcal{T}^*$ is formally calculated as usual. For $\Theta\in \mc D(\mc T)$ and $\overline\Theta$ smooth, integrations by parts yield
\begin{equation}\label{eq6.20}
\begin{split}
\int_{\Omega}\mathcal{T}\Theta\cdot\overline\Theta \text{d}x\text{d}y&=\int_\Omega \overline\Theta^t T_1\Theta_x + \overline\Theta^t T_2\Theta_y \text{d}x\text{d}y\\
&=-\int_\Omega \Theta^t T_1\overline\Theta_x + \Theta^t T_2\overline\Theta_y\, \text{d}x\text{d}y -\int_\Omega \Theta^t T_{1,x}\overline\Theta + \Theta^t T_{2,y}\overline\Theta \,\text{d}x\text{d}y \\
&\hspace{20pt}+ \int_0^{L_2} \overline\Theta^t T_1\Theta\Big|_{x=0}^{x=L_1} \text{d}y + \int_0^{L_1} \overline\Theta^t T_2\Theta\Big|_{y=0}^{y=L_2} \text{d}x \\
&=J_0 + J_1,
\end{split}
\end{equation}
where $J_0$ stands for the integral on $\Omega$ and $J_1$ for the integral on $\p\Omega$. For $J_0$, we have
\[
J_0 = \int_\Omega \mathcal{T}^*\overline\Theta\cdot\Theta\, \text{d}x\text{d}y,
\]
where $\mathcal{T}^*\overline\Theta=-T_1\overline\Theta_x-T_2\overline\Theta_y$.
For $J_1$, taking into account the boundary conditions \eqref{eq:eq6.1}, there remains:
\begin{equation}
\begin{split}
J_1&= \int_0^{L_2} (\alpha_1\bar\theta_1+\beta_1\bar\theta_2)\theta_1\Big|_{x=L_1}-(\beta_1\bar\theta_1-\alpha_1\bar\theta_2)\theta_2\Big|_{x=0}\text{d}y \\
&\hspace{20pt} +\int_0^{L_1} (\alpha_2\bar\theta_1+\beta_2\bar\theta_2)\theta_1\Big|_{y=L_2}-(\beta_2\bar\theta_1-\alpha_2\bar\theta_2)\theta_2\Big|_{y=0}\text{d}x.
\end{split}
\end{equation}

In order to guarantee that $\Theta\mapsto\inner{\mc T\Theta}{\overline \Theta}$ is continuous on $\mc D(\mc T)$ for the norm of $L^2(\Omega)^2$ (see Subsec. \ref{sec-supercritical-sub1}), $\mc T^*\overline\Theta$ must be in $L^2(\Omega)^2$ and the following boundary conditions must be satisfied:
\begin{equation}\label{eq:eq6.22}
\begin{cases}
\beta_1\bar\theta_1-\alpha_1\bar\theta_2 = 0, \text{ on } \Gamma_W=\{x=0\}, \\
\alpha_1\bar\theta_1+\beta_1\bar\theta_2 = 0, \text{ on } \Gamma_E=\{x=L_1\},\\
\beta_2\bar\theta_1-\alpha_2\bar\theta_2 = 0, \text{ on } \Gamma_S=\{y=0\}, \\
\alpha_2\bar\theta_1+\beta_2\bar\theta_2 = 0, \text{ on } \Gamma_N=\{y=L_2\}.
\end{cases}
\end{equation}
{
We conclude that
\begin{equation}\label{eq6.23}
\mc D(\mc T^*) \subset \set{ \overline\Theta=(\bar\theta_1,\bar\theta_2)^t\in L^2(\Omega)^2 \ |\ \mathcal{T^*} \overline\Theta\in L^2(\Omega)^2,\, \overline\Theta\text{ satisfies \eqref{eq:eq6.22} }}.
\end{equation}
We will explicitly prove below that $\mc D(\mc T^*)$ is indeed equal to the space in the right-hand side of \eqref{eq6.23} for completeness.
We remark that we do not need the equality here since the proof of Theorem~\ref{thm6.4} below is based on the existence of the solution to the boundary value problem for $\mc T$ in Theorem~\ref{thm6.3}.
\begin{thm}\label{thm6.4}
We assume that \eqref{eq:con6.1} holds. Then the operator $\mc T^*$ is positive.
\end{thm}
\begin{proof}
For any $\overline\Theta\in \mc D(\mc T^*)$, we find $\Theta\in V$ by Theorem \ref{thm6.3} satisfying
$\mc T\Theta  = \mc T^*\overline\Theta$. Hence, using the definition of the adjoint operator $\mc T^*$, we obtain
\begin{equation}\begin{split}
\inner{  \mc T^*\overline\Theta  }{\overline\Theta} = \inner{ \mc T\Theta  }{\overline\Theta} = \inner{ \Theta}{\mc T^*\overline\Theta  } = \inner{ \Theta}{\mc T\Theta }\geq 0,
\end{split}\end{equation}
which shows that $\mc T^*$ is positive.
\end{proof}

We now prove that $\mc D(\mc T^*)$ is indeed equal to the space in the right-hand side of \eqref{eq6.23}, which we temporarily denote by $\widetilde{\mathcal{D}}(\mathcal{T}^*)$.
We also introduce the function space
\begin{equation*}
\overline{V}=\{ \overline\Theta=(\bar\theta_1,\bar\theta_2)^t\in H^1(\Omega)^2 \ |\ \overline\Theta\text{ satisfies \eqref{eq:eq6.22} } \}.
\end{equation*}
It is clear that we have the following inclusions
\begin{equation}
  \overline{V} \subset \mc D(\mc T^*)   \subset \widetilde{\mathcal{D}}(\mathcal{T}^*),
\end{equation}
since the integrations by parts are valid in \eqref{eq6.20} for $H^1$-functions. We aim to show that
\begin{equation}\label{eq6.25}
  \overline{V} = \mc D(\mc T^*)   = \widetilde{\mathcal{D}}(\mathcal{T}^*),
\end{equation}
which is an immediate consequence of the following result.
\begin{thm}\label{thm6.6}
The space $\widetilde{\mathcal{D}}(\mathcal{T}^*)$ is the space $\overline V$.
\end{thm}
Noticing that Lemma \ref{lem6.1} is still valid for $\overline\Theta\in\overline V$, the proof of Theorem \ref{thm6.6} is the same as the proof of Theorem \ref{thm6.1} using the following proposition which replaces Proposition \ref{prop6.2}.
}

\begin{prop}\label{prop6.6}
For every given $\Psi=(\psi_1,\psi_2)^t\in \mathcal{D}(\Omega)^2$, the problem $\mathcal{T}^*\overline\Theta = \Psi$ possesses a {unique} solution $\overline\Theta\in \overline V$. {The solution of $\mathcal{T}^*\overline\Theta = \Psi$ is also unique in $\mc D(\mc T^*)$.}
\end{prop}
\begin{proof}
{
We first prove the uniqueness. We assume that $\mc T^*\overline\Theta=0$ for $\overline\Theta\in\overline V$, and we need to show that $\overline\Theta=0$. Since Lemma \ref{lem6.1} is still valid for $\overline\Theta\in\overline V$, we obtain that \eqref{eq:eq6.6} still holds for $\mc T^*\overline\Theta$, i.e.
\begin{equation}\label{eq6.15}
\f{1}{C_0}\norm{\nabla\overline\Theta}_{L^2(\Omega)^2}\leq \norm{\mathcal{T}^*\overline\Theta}_{L^2(\Omega)^2}\leq C_0 \norm{\nabla\overline\Theta}_{L^2(\Omega)^2}.
\end{equation}
We thus conclude from \eqref{eq6.15} that $\nabla\overline\Theta=0$, which implies that $\Theta$ is a constant function. The boundary conditions \eqref{eq:eq6.22} for $\Theta$ then impose that $\Theta\equiv 0$, and the uniqueness follows.

To prove the existence}, we introduce a suitable change of variables which makes the boundary conditions \eqref{eq:eq6.22} simpler. We set
\begin{equation*}
\Xi=\begin{pmatrix}
\xi_1 \\ \xi_2
\end{pmatrix}=\begin{pmatrix}
\alpha_1\bar\theta_1 + \beta_1\bar\theta_2 \\
\beta_1\bar\theta_1-\alpha_1\bar\theta_2
\end{pmatrix}=\begin{pmatrix}
\alpha_1 & \beta_1 \\
\beta_1 & -\alpha_1
\end{pmatrix}\begin{pmatrix}
\bar\theta_1 \\ \bar\theta_2
\end{pmatrix}=T_1\overline\Theta,
\end{equation*}
and find by direct computation 
\begin{equation*}
\begin{split}
\beta_2\bar\theta_1-\alpha_2\bar\theta_2 = \mu_1\xi_1+\mu_2\xi_2,\hspace{6pt}\alpha_2\bar\theta_1+\beta_2\bar\theta_2 = \mu_2\xi_1-\mu_1\xi_2,
\end{split}
\end{equation*}
where the constants $\mu_1$ and $\mu_2$ are given by
\begin{equation*}
\mu_1=\f{\alpha_1\beta_1+\alpha_2\beta_2}{\alpha_1^2+\beta_1^2},\quad\quad\mu_2=\f{\alpha_2\beta_1 - \alpha_1\beta_2}{\alpha_1^2+\beta_1^2}.
\end{equation*}
Thus the boundary conditions \eqref{eq:eq6.22} are equivalent to
\begin{equation}\tag{$\ref{eq:eq6.22}^\flat$}\label{eq:eq6.26}
\begin{cases}
\xi_2= 0, \text{ on } \Gamma_W=\{x=0\}, \\
\xi_1= 0, \text{ on } \Gamma_E=\{x=L_1\},\\
\mu_2\xi_1-\mu_1\xi_2= 0, \text{ on } \Gamma_S=\{y=0\},\\
\mu_1\xi_1+\mu_2\xi_2= 0, \text{ on } \Gamma_N=\{y=L_2\}.
\end{cases}
\end{equation}
With the new variables, we define the operator $\mathcal{P}$ such that $\mathcal{P}\Xi = \mathcal{T}^*\overline\Theta$, that is 
\begin{equation}
\begin{split}
\mathcal{P}\Xi=\mathcal{T}^*\overline\Theta &= -T_1\overline\Theta_x-T_2\overline\Theta_y= -\Xi - T_2T_1^{-1}\Xi \\
&=-\begin{pmatrix}
\xi_1 \\ \xi_2
\end{pmatrix}_x - \begin{pmatrix}
\mu_1 & \mu_2 \\
-\mu_2 & \mu_1
\end{pmatrix}\begin{pmatrix}
\xi_1 \\ \xi_2
\end{pmatrix}_y\\
&=-\begin{pmatrix}
1 & \mu_1 \\
0 & -\mu_2
\end{pmatrix}\nabla\xi_1 -\begin{pmatrix}
0 & \mu_2 \\
1 & \mu_1
\end{pmatrix}\nabla\xi_2. \\
\end{split}
\end{equation}
The goal now becomes the following: given $\Psi=(\psi_1,\psi_2)^t\in\mathcal{D}(\Omega)^2$, we look for $\Xi\in \overline{V}^\flat$ such that $\mathcal{P}\Xi = \Psi$, where the function space $\overline{V}^\flat$ is defined as
\begin{equation*}
\overline{V}^\flat=\{ \Xi=(\xi_1,\xi_2)^t\in H^1(\Omega)^2 \ |\ \Xi\text{ satisfies \eqref{eq:eq6.26} } \}.
\end{equation*} 

Noticing that $\mu_2$ is nonzero because of assumption \eqref{eq:con6.1}, we can assume without loss of generality that $\mu_2=1$. Following the proof of Proposition \ref{prop6.2}, we then introduce a new coordinate system $(x',y')$ such that
\begin{equation*}
\begin{pmatrix}
x' \\ y'
\end{pmatrix}=\begin{pmatrix}
1 & 0 \\
-\mu_1 & 1
\end{pmatrix}\begin{pmatrix}
x \\ y
\end{pmatrix},\hspace{6pt}
\begin{pmatrix}
x \\ y
\end{pmatrix}=\begin{pmatrix}
1 & 0 \\
\mu_1 & 1
\end{pmatrix}\begin{pmatrix}
x' \\ y'
\end{pmatrix}.
\end{equation*}
We denote by $\nabla'$ the gradient in the new coordinate system $(x',y')$
In the new coordinate system $(x',y')$, the gradient $\nabla'$ and the operator $\mathcal{P}$ read
\begin{equation}
\nabla\xi=\begin{pmatrix}
1 & -\mu_1 \\
0 & 1
\end{pmatrix}\nabla'\xi;
\end{equation}
\begin{equation}
\begin{split}
\mathcal{P}\Xi&=-\begin{pmatrix}
1 & \mu_1 \\
0 & -1
\end{pmatrix}\nabla\xi_1 -\begin{pmatrix}
0 & 1 \\
1 & \mu_1
\end{pmatrix}\nabla\xi_2 \\
&=-\begin{pmatrix}
1 & \mu_1 \\
0 & -1
\end{pmatrix}\begin{pmatrix}
1 & -\mu_1 \\
0 & 1
\end{pmatrix}\nabla'\xi_1 -\begin{pmatrix}
0 & 1 \\
1 & \mu_1
\end{pmatrix}\begin{pmatrix}
1 & -\mu_1 \\
0 & 1
\end{pmatrix}\nabla'\xi_2 \\
&=\begin{pmatrix}
-\xi_{1x'}-\xi_{2y'} \\
\xi_{1y'}-\xi_{2x'}
\end{pmatrix};
\end{split}
\end{equation}
and the boundary conditions \eqref{eq:eq6.26} read
\begin{equation}\tag{$\ref{eq:eq6.22}^\sharp$}\label{eq:eq6.29}
\begin{cases}
\xi_2= 0, \text{ on } \Gamma_W', \\
\xi_1= 0, \text{ on } \Gamma_E',\\
\xi_1-\mu_1\xi_2= 0, \text{ on } \Gamma_S', \\
\mu_1\xi_1+\xi_2= 0, \text{ on } \Gamma_N'.
\end{cases}
\end{equation}
Finally we define the function space
\begin{equation*}
\overline{V}^\sharp=\{ \Xi=(\xi_1,\xi_2)^t\in H^1(\Omega')^2 \ |\ \Xi\text{ satisfies \eqref{eq:eq6.29} } \}.
\end{equation*}
We are now seeking $\Xi\in\overline{V}^\sharp$ satisfying $\mathcal{P}\Xi=\Psi$, that is
\begin{equation}\label{eq:eq6.30}
\begin{cases}
-\xi_{1x'}-\xi_{2y'} =\psi_1, \\
\xi_{1y'}-\xi_{2x'} =\psi_2.
\end{cases}
\end{equation}
We now have two cases to consider. First, if $\mu_1=0$, then the boundary conditions \eqref{eq:eq6.29} are simpler, and similar to the boundary conditions \eqref{eq:eq6.1p}. Hence, we obtain a unique solution $\Xi\in \overline{V}^\flat$ that solves the problem $\mathcal{P}\Xi=\Psi$ by following the same arguments as for Proposition \ref{prop6.2}. If $\mu_1\neq 0$, then without loss of generality, we can assume that $\mu_1=1$, and the arguments presented below will be similar to the proof of Proposition \ref{prop6.2}. 

Differentiating $\eqref{eq:eq6.30}_1$ with respect to $y'$ and $\eqref{eq:eq6.30}_2$ with respect to $x'$, and summing these two equations, we find
\begin{equation}\label{eq:eq6.31}
\Delta'\xi_2 = -\psi_{1y'}-\psi_{2x'}.
\end{equation}
We first associate with equation \eqref{eq:eq6.31} the boundary condition
\begin{equation}
\xi_2 = 0, \text{ on } \Gamma_W',
\end{equation}
which is already contained in \eqref{eq:eq6.29}. For the other boundaries, suitable boundary conditions can be obtained as follows. As in Proposition \ref{prop6.2}, we denote by $\nu_i'$ the unit normal vector to $\Gamma_i'$, and $\tau_i'$ the unit tangent vector on $\Gamma_i'$, for all $i\in\{W,E,S,N\}$.
On $\Gamma_E'$, we have $\p\xi_1/\p\tau_2' =0$ since $\xi_1=0$ on $\Gamma_E'$. Noticing that $(0,1)$ is parallel to $\tau_2'$, we find that $\xi_{1y'}=0$, which, together with $\eqref{eq:eq6.30}_2$, implies that $\xi_{2x'} = -\xi_{1y'}=0$ by noticing that $\psi_2$ vanishes on $\Gamma_E'$.
Observing that $\nu_E'$ is parallel to $(1,0)$, we then associate to \eqref{eq:eq6.31} the following boundary condition on $\Gamma_E'$:
\begin{equation}
\f{\p\xi_2}{\p\nu_E'} =0, \text{ on } \Gamma_E'.
\end{equation}
Similar arguments show that we need to associate to \eqref{eq:eq6.31} the following condition on the other two boundaries:
\begin{equation}\label{eq:eq6.34}
\begin{cases}
\f{\p\xi_2}{\p\nu_S'} - \f{\p\xi_2}{\p\tau_S'} =0,  \text{ on } \Gamma_S', \\
\f{\p\xi_2}{\p\nu_N'} + \f{\p\xi_2}{\p\tau_N'} =0,  \text{ on } \Gamma_N'. \\ 
\end{cases}
\end{equation}
The existence of a (possibly non-unique) solution $\xi_2\in H^1(\Omega')$ of \eqref{eq:eq6.31}-\eqref{eq:eq6.34} follows from Lemma 4.4.4.2 of \cite{Gri85} with $\Omega=\Omega', \mathscr{D} = \Gamma_W', \mathscr{N} = \Gamma_E'\cup\Gamma_S'\cup\Gamma_N', f = -\psi_{1y'}-\psi_{2x'}$. 
Integrating $\eqref{eq:eq6.31}_2$ with respect to $y'$ and using the boundary condition \eqref{eq:eq6.29}$_3$, we obtain a solution $\xi_1$ which belongs to $H^1(\Omega')$.

In conclusion, in the new coordinate system $(x',y')$, we find a (possibly non-unique) solution $\Xi$ that belongs to $\overline{V}^\sharp$, and solves the problem $\mathcal{P}\Xi=\Psi$. Transforming back to the original system $(x,y)$ and the original variables $\Theta$, we obtain $\overline\Theta\in\overline V$ satisfying $\mathcal{T}^*\overline\Theta=\Psi$. 

{The proof of the uniqueness in $\mc D(\mc T^*)$ is the same as in the proof of Proposition \ref{prop6.2}.
Therefore, the proof of Proposition \ref{prop6.6} is complete.}
\end{proof}

\subsection{The mixed subcritical case}\label{sec-mixed-subcritical}
{We are now ready to study the stationary 2D shallow water equations operator for the elliptic-hyperbolic case when $\Delta=u_0^2+v_0^2-g\phi_0<0$, that is
\begin{equation}\label{eq:asp7.1}
u_0^2+v_0^2 < g\phi_0, \hspace{5pt}\text{implying}\hspace{5pt}u_0^2<g\phi_0,\hspace{3pt} v_0^2<g\phi_0.
\end{equation}
 In order to specify the boundary conditions, we recall the decoupled system \eqref{eq:eq2.8}. For $\zeta$, we need to assign the boundary conditions at $x=0$ and $y=0$ since $u_0$ and $v_0$ are both positive. For $\xi$ and $\eta$, as we see in Remark \ref{rmk6.0} and Theorem \ref{thm6.2}, by taking into account the energy estimates, we can assign different kinds of boundary conditions for $\xi$ and $\eta$ which are suitable for the  well-posedness.
Here, we assign the boundary conditions for $\xi$ at $x=0$ and $y=0$, and the boundary conditions for $\eta$ at $x=L_1$ and $y=L_2$.
In conclusion, the boundary conditions for $U$ are}
\begin{equation}\label{eq:eq7.1}
\begin{cases}
\xi=v_0u - u_0v=0,\; \zeta = u_0u + v_0v + g\phi = 0, \text{ on } \Gamma_W\cup\Gamma_S=\{x=0\}\cup\{y=0\}, \\
\eta\simeq\phi = 0, \text{ on } \Gamma_E\cup\Gamma_N=\{x=L_1\}\cup\{y=L_2\}.
\end{cases}
\end{equation}
We then define the function space
\begin{equation*}
\mathcal{D}(A) = \{ U\in H = L^2(\Omega)^3,\quad \mathcal{A}U\in H,\text{ and }  U\text{ satisfies }\eqref{eq:eq7.1} \},\\
\end{equation*}
and set $AU=\mathcal{A}U,\,\forall U\in\mathcal{D}(A)$. 

\subsubsection{Positivity of $A$ and its adjoint $A^*$}\label{subsec7.1}
We use the notations {$\xi,\eta,\zeta$ and $P$ defined in \eqref{eqe.5}} and set $\Xi=(\xi,\eta,\zeta)^t$, so that $U=P\Xi$.  {With the diagonalization \eqref{eqe.2}, we compute for $U\in\mathcal{D}(A)$:}
\begin{equation}\label{eq:eq7.3}
\begin{split}
\inner{ AU}{U}_H &=\int_\Omega (u_0u_x + v_0u_y + g\phi_x)u + 
(u_0v_x + v_0v_y + g\phi_y)v \\
&\hspace{20pt}+ \f{g}{\phi_0}(u_0\phi_x + v_0\phi_y + 
\phi_0(u_x+v_y))\phi \text{d}x\text{d}y \\
&=\inner{S_0\mathcal{E}_1U_x + S_0\mathcal{E}_2U_y}{U}_H  \\
&=\inner{P^tS_0\mathcal{E}_1P\Xi_x + P^tS_0\mathcal{E}_2P\Xi_y}{\Xi}_H \\
&=\int_\Omega 
(\xi,\eta)\f{1}{u_0^2+v_0^2}\bigg(\begin{pmatrix}
u_0&\f{gv_0}{\kappa_1}\\
\f{gv_0}{\kappa_1}&-u_0\\
\end{pmatrix}
\begin{pmatrix}
\xi \\ \eta
\end{pmatrix}_x
+
\begin{pmatrix}
v_0&-\f{gu_0}{\kappa_1}\\
-\f{gu_0}{\kappa_1}&-v_0\\
\end{pmatrix}
\begin{pmatrix}
\xi \\ \eta 
\end{pmatrix}_y\bigg)
\text{d}x\text{d}y \\
&\hspace{20pt}+\f{1}{u_0^2+v_0^2}\int_\Omega \zeta(u_0\zeta_x+v_0\zeta_y) \text{d}x\text{d}y. \\
\end{split}
\end{equation}
{
We now need to show that $J_0,J_1$ in the right-hand side of \eqref{eq:eq7.3} are both non-negative. According to Theorem \ref{thm6.2} with 
\begin{equation}\label{eq:eq7.10}
  \mc T\begin{pmatrix}
\xi \\ \eta 
\end{pmatrix} = \f{1}{u_0^2+v_0^2}\begin{pmatrix}
u_0&\f{gv_0}{\kappa_1}\\
\f{gv_0}{\kappa_1}&-u_0\\
\end{pmatrix}
\begin{pmatrix}
\xi \\ \eta
\end{pmatrix}_x +
\f{1}{u_0^2+v_0^2}\begin{pmatrix}
v_0&-\f{gu_0}{\kappa_1}\\
-\f{gu_0}{\kappa_1}&-v_0\\
\end{pmatrix}
\begin{pmatrix}
\xi \\ \eta 
\end{pmatrix}_y,
\end{equation}
we find that $J_0\geq 0$.
For $J_1$, by density Theorem \ref{thm1} applied with $\theta=\zeta,\lambda=u_0/v_0,\Gamma=\Gamma_W\cup\Gamma_S$, we see that the integration by parts is valid for $J_1$, and we find
\begin{equation*}
J_1=\f{1}{2(u_0^2+v_0^2)} \Big(\int_0^{L_2}u_0\zeta^2 \Big |_{x=0}^{x=L_1}\text{d}y + \int_0^{L_1} v_0\zeta^2 \Big|_{y=0}^{y=L_2} \text{d}x \Big)\geq 0,
\end{equation*}
where the last inequality follows from the boundary conditions for $\zeta=u_0u + v_0v + g\phi$ and the assumption $u_0,v_0>0$.

Combining the results for $J_0$ and $J_1$, we obtain that $\inner{ AU}{U}_H\geq 0$.
Note that here, because the part of the operator $A$ which is elliptic gives some additional regularity, we did not need to use an approximation argument to show that $\inner{ AU}{U}_H\geq 0,\,\forall U\in\mathcal{D}(A)$, unlike in the previous cases.
}

We now turn to the adjoint $A^*$. In the mixed subcritical case, the formal definition of $A^*$ can also be treated similarly as in Subsection \ref{sec-supercritical-sub1}, we thus omit the details. The adjoint differential operator $\mathcal{A}^*$ is given by \eqref{eq:eq3.3} again.
In order to guarantee that $U\mapsto\inner{ AU}{\overline U}_H$ is continuous on $\mathcal{D}(A)$, the following boundary conditions in the $(\xi,\eta,\zeta)$ variables need to be satisfied (see \eqref{eq:eq6.22} since the equations for $\xi,\eta$ are elliptic):
\begin{equation}
\begin{cases}
\f{gv_0}{\kappa_1}\bar\xi - u_0\bar\eta =0, \text{ on } \Gamma_W=\{x=0\}, \\
\bar\zeta = u_0\bar\xi+\f{gv_0}{\kappa_1}\bar\eta=0,\text{ on }\Gamma_E=\{x=L_1\},\\
-\f{gu_0}{\kappa_1}\bar\xi -v_0\bar\eta=0,\text{ on }\Gamma_S=\{y=0\},\\
\bar\zeta=v_0\bar\xi-\f{gu_0}{\kappa_1}\bar\eta=0,\text{ on }\Gamma_N=\{y=L_2\}.
\end{cases}
\end{equation}
Transforming back to the $(u,v,\phi)$ variables, the boundary conditions read
\begin{equation}\label{eq:eq7.4}
\begin{cases}
gv_0^2\bar u - gv_0u_0\bar v - u_0\kappa_1^2\bar\phi =0, \text{ on } \Gamma_W=\{x=0\}, \\
u_0v_0\bar u-u_0^2\bar v+gv_0\bar\phi=u_0\bar u+v_0\bar v+g\bar\phi=0, \text{ on }\Gamma_E=\{x=L_1\},\\
gu_0v_0\bar u-gu_0^2\bar v+v_0\kappa_1^2\bar\phi=0, \text{ on }\Gamma_S=\{y=0\},\\
v_0^2\bar u-v_0u_0\bar v - gu_0\bar\phi=u_0\bar u+v_0\bar v+g\bar\phi=0, \text{ on }\Gamma_N=\{y=L_2\},
\end{cases}
\end{equation}
which are more complicated than the boundary conditions \eqref{eq:eq7.1}, since some elliptic modes are hidden in the operator $A$.
{We conclude that (see also \eqref{eq6.25})}
\begin{equation*}
\mathcal{D}(A^*) = \{ \overline{U} \in H = L^2(\Omega)^3,\quad A^*\overline{U}\in H, \text{ and }\overline{U}\text{ satisfies \eqref{eq:eq7.4}}  \},
\end{equation*}
and set $A^*\overline U=\mathcal{A}^*\overline U,\,\forall\overline U\in\mathcal{D}(A^*)$.

{
The proof of the positivity of $A^*$ is similar to that of $A$ where we need to use Theorem~\ref{thm6.4} for the elliptic mode hidden in the operator $A^*$, we thus omit it here.
}

\section{Existence and uniqueness results and additional remarks}\label{sec-full-system}
In this section we aim to combine the results of the previous sections and to investigate the well-posedness for Eqs. \eqref{eq:eq1.1} associated with the suitable initial and boundary conditions that we have introduced. Here again, we consider the case of homogeneous boundary conditions, and give below a remark about the non-homogeneous boundary conditions.

{
As we have already seen in the previous sections, we have two cases to consider depending on the sign of $\Delta=u_0^2+v_0^2-g\phi_0$, and there are four sub-cases for the case when $\Delta=u_0^2+v_0^2-g\phi_0$ is positive.
In the following, we list these cases and their corresponding boundary conditions.}

{We first consider \textsc{the fully hyperbolic case}, hence we assume that 
$$\Delta=u_0^2+v_0^2-g\phi_0>0,$$
and as we have already seen, this leads to four sub-cases:}

\quad\quad\textbf{The supercritical case. }The assumption is
\begin{equation}\label{eq:eq8.1}
u_0^2>g\phi_0,\hspace{6pt}v_0^2>g\phi_0,
\end{equation}
and the corresponding boundary conditions are
\begin{equation}\tag{\ref{eq:eq8.1}$'$}\label{eq:eq8.1p}
\begin{cases}
u = v = \phi = 0, \text{ on } \Gamma_W=\{x=0\}, \\
u = v = \phi = 0, \text{ on } \Gamma_S=\{y=0\};
\end{cases}
\end{equation}

\quad\quad\textbf{The mixed hyperbolic case I. }The assumption is
\begin{equation}\label{eq:eq8.2}
u_0^2<g\phi_0,\hspace{6pt}v_0^2>g\phi_0,
\end{equation}
and the corresponding boundary conditions are
\begin{equation}\tag{\ref{eq:eq8.2}$'$}\label{eq:eq8.2p}
\begin{cases}
v_0u - u_0v + \kappa_0 \phi = u_0u + v_0v + g\phi=0,\text{ on } \Gamma_W=\{x=0\}, \\
v_0u - u_0v - \kappa_0 \phi = 0, \text{ on } \Gamma_E=\{x=L_1\}, \\
u = v = \phi = 0, \text{ on } \Gamma_S=\{y=0\};
\end{cases}
\end{equation}

\quad\quad\textbf{The mixed hyperbolic case II. }The assumption is
\begin{equation}\label{eq:eq8.3}
u_0^2>g\phi_0,\hspace{6pt}v_0^2<g\phi_0,
\end{equation}
and the corresponding boundary conditions are
\begin{equation}\tag{\ref{eq:eq8.3}$'$}\label{eq:eq8.3p}
\begin{cases}
v_0u - u_0v - \kappa_0 \phi =u_0u + v_0v + g\phi  =0,\text{ on } \Gamma_S=\{y=0\}, \\
v_0u - u_0v + \kappa_0 \phi = 0, \text{ on } \Gamma_N=\{y=L_2\}, \\
u = v = \phi = 0, \text{ on } \Gamma_W=\{x=0\};
\end{cases}
\end{equation}

\quad\quad\textbf{The fully hyperbolic subcritical case. }The assumption is
\begin{equation}\label{eq:eq8.4}
u_0^2<g\phi_0,\hspace{6pt}v_0^2<g\phi_0,\hspace{6pt}u_0^2+v_0^2>g\phi_0,
\end{equation}
and the corresponding boundary conditions are
\begin{equation}\tag{\ref{eq:eq8.4}$'$}\label{eq:eq8.4p}
\begin{cases}
v_0u - u_0v + \kappa_0 \phi = u_0u + v_0v + g\phi = 0, \text{ on } \Gamma_W=\{x=0\}, \\
v_0u - u_0v - \kappa_0 \phi = 0, \text{ on } \Gamma_E=\{x=L_1\}, \\
v_0u - u_0v - \kappa_0 \phi = u_0u + v_0v + g\phi = 0, \text{ on } \Gamma_S=\{y=0\}, \\
v_0u - u_0v + \kappa_0 \phi = 0, \text{ on } \Gamma_N=\{y=L_2\}.
\end{cases}
\end{equation}

{
We then consider \textsc{the elliptic-hyperbolic case} where
\[
\Delta =u_0^2+v_0^2-g\phi_0<0,
\]
and this leads to only one case, which we also call the mixed subcritical case.}

\quad\quad\textbf{The mixed subcritical case. }The assumption is
\begin{equation}\label{eq:eq8.5}
u_0^2<g\phi_0,\hspace{6pt}v_0^2<g\phi_0,\hspace{6pt}u_0^2+v_0^2<g\phi_0,
\end{equation}
and the corresponding boundary conditions are
\begin{equation}\tag{\ref{eq:eq8.5}$'$}\label{eq:eq8.5p}
\begin{cases}
v_0u - u_0v  = u_0u + v_0v + g\phi = 0, \text{ on } \Gamma_W\cup\Gamma_S=\{x=0\}\cup\{y=0\}, \\
\phi = 0, \text{ on } \Gamma_E\cup\Gamma_N=\{x=L_1\}\cup\{y=L_2\}.
\end{cases}
\end{equation}

If these constants $u_0,v_0,\phi_0,g$ satisfy assumption $(\thesection.j)$, we then define the domain of the unbounded operator $A$: 
\begin{equation*}
\begin{split}
\mathcal{D}(A) = \{ U\in H = L^2(\Omega)^3,\quad AU\in H\text{ and }  U\text{ satisfies }(\thesection.j') \},
\end{split}
\end{equation*}
where $j\in\{1,2,3,4,5\}$, indicating the five different cases. 

We set $BU=(-fv,fu,0)^t$, where $f$ is the Coriolis parameter; it is easy to see that $B$ is a linear continuous {anti-self-adjoint (i.e. $B^*=-B$)} operator on $H$. We also set $A_0=A+B$, with 
$\mathcal{D}(A_0)=\mathcal{D}(A)$, and consider the adjoint $A_0^*=A^*+B^*$, with $\mathcal{D}(A_0^*)=\mathcal{D}(A^*)$. Then we have
\begin{thm}\label{thm8.1}
The operator $-A_0$ is the infinitesimal generator of a contraction semigroup on $H$. 
\end{thm}
\begin{proof}
According to \cite{Yos80, HP74}, it suffices to show that
\begin{enumerate}[(i)]
\item $A_0$ and $A_0^*$ are both closed operators, and their domains $\mathcal{D}(A_0)$ and $\mathcal{D}(A_0^*)$ are dense in $H$.
\item $A_0$ and $A_0^*$ are both positive.
\end{enumerate}
{Noticing that $B$ is a linear continuous operator on $H$ and that $\inner{BU}{U}=0$ and $\inner{B^*U}{U}=0$ for all $U\in H$,} we find that proving (i) and (ii) is equivalent to proving the following:
\begin{enumerate}[(i$'$)]
\item $A$ and $A^*$ are both closed operators, and their domains $\mathcal{D}(A)$ and $\mathcal{D}(A^*)$ are dense in $H$.
\item $A$ and $A^*$ are both positive.
\end{enumerate}

We already showed that $A$ and $A^*$ are both positive in the previous subsections,
we thus only need to prove (i$'$). We establish the result for $A$, and the proof for $A^*$ would be similar.

In all cases, observing that $\mathcal{D}(\Omega)^3$ is included in $\mathcal{D}(A)$ and dense in $H=L^2(\Omega)^3$, we thus obtain that $\mathcal{D}(A)$ is dense in $H$.
To show that $A$ is closed, consider a sequence $\{U_n\}_{n\in\mathbb{N}}\subset \mathcal{D}(A)$ for which $\lim_{n\rightarrow\infty}U_n=U$ in $L^2(\Omega)^3$ and $\lim_{n\rightarrow\infty}\mc AU_n=F$ in $L^2(\Omega)^3$. By the $L^2$-convergence of $\{U_n\}_{n\in\mathbb{N}}$, we first find that $\mc AU_n$ converges to $\mc AU$ in the sense of distributions, which implies that $\mc AU$ equals $F$ in the sense of distributions. Since $F$ belongs to $L^2(\Omega)^3$, we obtain that $\mc AU$ belongs to $L^2(\Omega)^3$, too. Thus, we have $U, \mc AU\in L^2(\Omega)^3$, which shows that the traces of $U$ are well-defined thanks to Proposition \ref{prop2.1}. By Proposition \ref{prop2.1} again, the traces of $U_n$ converge weakly to the traces of $U$ in the appropriate space $H^{-1}$, so that $U$ satisfies the boundary conditions $(\thesection.j')$, where $j\in\{1,2,3,4,5\}$. Therefore, we conclude that $U\in\mathcal{D}(A)$. Hence $A$ is closed, and the proof is complete. 
\end{proof}

We now consider the whole system of 2D linearized Shallow Water Equations, namely \eqref{eq:eq1.1} and introduce the initial and boundary conditions.
As we discussed before, the boundary conditions are $(\thesection.j'), j\in\{1,2,3,4,5\}$ depending on the case we are considering, and all these boundary conditions are taken into account in the domain $\mathcal{D}(A_0)$ of $A_0$. Finally if we add the initial conditions:
\begin{equation}\label{eq:eq8.8}
U(0)=(u(0),v(0),\phi(0))=U^0=(u^0,v^0,\phi^0),
\end{equation}
then the initial and boundary value problem consisting of Eqs. \eqref{eq:eq1.1}, $(\thesection.j'), j\in\{1,2,3,4,5\}$ and \eqref{eq:eq8.8} is equivalent to the abstract initial value problem
\begin{equation}\label{eq:eq8.9}
\begin{cases}
\f{\text{d}U}{\text{d}t} + A_0U = F, \\
U(0)=U^0.
\end{cases}
\end{equation}
Note that $F=(F_u,F_v,F_\phi)$ which does not appear in \eqref{eq:eq1.1} is added here for mathematical generality and to study the case of non-homogeneous boundary conditions. Thanks to Theorem \ref{thm8.1} this problem is now solved by the Hille-Yoshida theorem and we have:
\begin{thm}\label{thm8.2}
Let $H, A_0$ and $\mathcal{D}(A_0)$ be defined as before. Then the initial value problem \eqref{eq:eq8.9} is well-posed. That is, 
\begin{enumerate}[i)]

\item for every $U^0\in H$, and $F\in L^1(0,T;H)$, the problem \eqref{eq:eq8.9} admits a unique weak solution $U\in \mathcal{C}([0,T];H)$ satisfying
\begin{equation*}
U(t)=S(t)U^0 + \int_0^t S(t-s)F(s)\text{\emph d}s,\,\forall t\in[0,T],
\end{equation*}
where $(S(t))_{t\geq 0}$ is the contraction semigroup generated by the operator $-A_0$;

\item for every $U^0\in\mathcal{D}(A_0)$, and $F\in L^1(0,T;H)$, with $F'= \text{\emph d}F/\text{\emph d}t\in L^1(0,T;H)$, the problem \eqref{eq:eq8.9} has a unique strong solution $U$ such that
{
\begin{equation*}
U\in \mathcal{C}\big([0,T];\mathcal{D}(A_0) \big),\hspace{6pt} \f{\text{\emph d}U}{\text{\emph d}t}\in \mathcal{C}\big([0,T];H\big).
\end{equation*}
}
\end{enumerate}

\end{thm}

{
We conclude with three remarks concerning the extension of this work to other domains,  the connections with the Primitive Equations of the atmosphere and the oceans and with the issue of non-homogeneous boundary conditions.

\begin{rmk}\label{rmk8.1}
It is likely  that much of the results concerning the fully hyperbolic case can be extended to more general convex polygonal domains instead of a rectangle, or probably more general polygonal-like domains. However, in the part concerning the elliptic-hyperbolic case, we use some of Grisvard's results which require very restrictive conditions on the angles. Also we preferred to consider a rectangle domain to stay close from the initial motivation of this study in LAMs. 
\end{rmk}
}

\begin{rmk}\label{rmk8.2}
The linearized 3D inviscid Primitive Equations of the atmosphere and the oceans can be written:
\begin{equation}\label{eq:eq8.11}
\begin{cases}
u_t + U_0u_x + V_0u_y - fv + \phi_x = 0, \\
v_t + U_0v_x + V_0v_y + fu + \phi_y = 0, \\
T_t + U_0T_x + V_0T_y + N^2\f{T_0}{g}w = 0, \\
u_x + v_y + w_z = 0, \\
\phi_z = \f{gT}{T_0},
\end{cases}
\end{equation}
where $(u,v,w)$ is the velocity of the water, $(u,v)$ the horizontal velocity, $T$ the temperature, $\phi$ a multiple of the pressure, $g$ the gravitational acceleration, and $N^2$ denotes the Brunt-V$\ddot{\text{a}}$is$\ddot{\text{a}}$l$\ddot{\text{a}}$ (buoyancy) frequency satisfying
\begin{equation*}
N^2 = -\f{g}{\rho_0}\f{d\rho}{dz},
\end{equation*}
and $U_0>0,V_0>0,\rho_0>0$ and $T_0>0$ are reference average values of the density and the temperature. 

We refer to \cite{RTT08b} for a systematic discussion in the case when $V_0=0$. Using the normal modes expansion for \eqref{eq:eq8.11} (see details in \cite{RTT08b}), we obtain the equations for the non-zero modes $(u_n,v_n,\phi_n, \lambda_n)$ which read (the indices $n$ are dropped for the sake of simplicity):
\begin{equation}\label{eq:eq8.12}
\begin{cases}
u_t + U_0u_x + V_0u_y - fv - \f{1}{\lambda}\psi_x = 0,\\
v_t + U_0v_x + V_0v_y + fu - \f{1}{\lambda}\psi_y = 0, \\
\psi_t + U_0\psi_x + V_0\psi_y - \f{N^2}{\lambda}(u_x + v_y) = 0,
\end{cases}
\end{equation}
where {$\psi=-\lambda\phi$} and $\lambda$ is positive. Observing that \eqref{eq:eq8.12} has the same form as the shallow water equation \eqref{eq:eq1.1} if we replace $\phi$ and $g$ in \eqref{eq:eq1.1} by $-\psi$ and $\lambda^{-1}$, we thus can obtain a well-posedness result (see Theorem \ref{thm8.2}) for \eqref{eq:eq8.12}. The equations for the zero mode are 
\begin{equation}\label{eq:eq8.13}
\begin{cases}
u_t + U_0u_x + V_0u_y - fv + \f{1}{\lambda}\phi_x = 0,\\
v_t + U_0v_x + V_0v_y + fu + \f{1}{\lambda}\phi_y = 0, \\
u_x + v_y = 0.
\end{cases}
\end{equation}
Following the same argument as for the zero mode in \cite{RTT08b} and \cite{CST10}, we can also obtain the well-posedness result for \eqref{eq:eq8.13}. Therefore, we can obtain a well-posedness result for the whole system \eqref{eq:eq8.11} with suitable boundary conditions and initial conditions as in Section 4 of \cite{RTT08b}. The details will appear elsewhere.

\end{rmk}

\begin{rmk}

{

In this remark, we briefly show how our results can be extended to the case of non-homogeneous boundary conditions, that is we want to solve \eqref{eq:eq1.1} with $(\thesection.j')$ ($j\in\set{1,2,3,4,5}$) in which the boundary conditions are now non-homogeneous, and with initial condition \eqref{eq:eq8.8}. We write the problem as follows
\begin{equation}
\begin{cases}
U_t + \mathcal{A}U + BU = F,\\
U(0)=U^0,\\
+\text{suitable non-homogeneous boundary conditions}.
\end{cases}
\end{equation}
We assume that the boundary data are inferred from a function $U^g$ which is defined on $\Omega\times[0,T]$, and we remark that this assumption is reasonable since we have a lifting result for the domain $\Omega$ (see Lemma 1.5.2.3 in \cite{Gri85}). We now set
\[
U = U' + U^g.
\]
Then $U'$ will be sought as the solution of the problem:
\begin{equation}\label{eq10.8}
\begin{cases}
U'_t + \mathcal{A}U' + BU' = F', \\
 U'(0) = U^0 - U^g|_{t=0},\\
 +\text{suitable homogeneous boundary conditions},
 \end{cases}
\end{equation}
where
\[
F' = F - U^g_t - \mathcal{A}U^g - BU^g.
\]
Observing that $U'\in\mathcal{D}(A)$ if $U'$ is smooth enough (suitable homogeneous boundary conditions), we then rewrite \eqref{eq10.8} as the linear evolution equation
\begin{equation}\label{eq10.9}
\begin{cases}
\f{\text{\emph d}U'}{\text{\emph d}t} + A_0U' = F',\\
 U'(0) = U^0 - U^g|_{t=0},
 \end{cases}
\end{equation}
where $A_0 = A + B$ with $\mathcal{D}(A_0)=\mathcal{D}(A)$. 
We already knew that $A_0$ generates a contraction semigroup on $H$ (see Theorem~\ref{thm8.1}). Therefore, if we assume that
\begin{equation}\label{eq10.10}
F,\, U^g,\,U^g_t,\, U^g_x,\, U^g_y \in L^1(0,T; H),\quad U'(0)\in H,
\end{equation}
then the problem \eqref{eq10.9} admits a unique weak solution $U'\in\mathcal{C}([0,T];H)$ by the Hille-Yoshida theorem. We could also obtain a strong solution $U'$ of \eqref{eq10.9} if we assume more regularity on the data $F,\,U^g,\,U^0$, and we omit the details here. Typically we would supplement \eqref{eq10.10} with
\begin{equation}
  F_t,\, U_t^g,\,U^g_{tt},\, U^g_{xt},\, U^g_{yt} \in L^1(0,T; H),
\end{equation}
and
\begin{equation}
  U'(0)\in\mc D(A_0).
\end{equation}
Note that the later condition contains the usual compatibility conditions at $t=0$ between $U^0|_{\p\Omega}$ and $U^g(t=0)|_{\p\Omega}$ (see \cite{RM74}).
}

\end{rmk}

\section*{Acknowledgments}
This work was partially supported by the National Science Foundation under the grants NSF DMS-0906440 and DMS-1206438, and by the Research Fund of Indiana University. 
The authors thank the anonymous referees for their very useful remarks.

\bibliographystyle{amsalpha}

\begin{thebibliography}{RTT08b}

\bibitem[AB03]{AB03}
E.~Audusse and M.-O. Bristeau, \emph{Transport of pollutant in shallow water:
  {A} two time steps kinetic method}, M2AN Math. Model. Numer. Anal.
  \textbf{37} (2003), no.~2, 389--416.

\bibitem[ABBKP]{ABBKP}
E.~Audusse, F.~Bouchut, M.-O. Bristeau, R.~Klein, and B.~Perthame, \emph{A fast
  and stable well-balanced scheme with hydrostatic reconstruction for shallow
  water flows}, SIAM Journal of Scientiﬁc Computing \textbf{25} (2004),
  no.~6, 2050--2065.

\bibitem[ABPS]{ABPS}
E.~Audusse, M.O. Bristeau, B.~Perthame, and J.~Sainte-Marie, \emph{A multilayer
  saint-venant system with mass exchanges for shallow water flows. derivation
  and numerical validation}, M2AN Math. Model. Numer. Anal. \textbf{45 (1)}
  (2011), 169--200.

\bibitem[BC01]{BC01}
Marie-Odile Bristeau and Benoit Coussin, \emph{{Boundary Conditions for the
  Shallow Water Equations solved by Kinetic Schemes}}, no.~RR-4282, Projet M3N.

\bibitem[BD03]{BD03}
Didier Bresch and Beno{\^{\i}}t Desjardins, \emph{Existence of global weak
  solutions for a 2{D} viscous shallow water equations and convergence to the
  quasi-geostrophic model}, Comm. Math. Phys. \textbf{238} (2003), no.~1-2,
  211--223.

\bibitem[BDM07]{BDM07}
Didier Bresch, Beno{\^{\i}}t Desjardins, and Guy M{\'e}tivier, \emph{Recent
  mathematical results and open problems about shallow water equations},
  Analysis and simulation of fluid dynamics, Adv. Math. Fluid Mech.,
  Birkh\"auser, Basel, 2007, pp.~15--31.

\bibitem[BS07]{BS07}
S.~Benzoni-Gavage and D.~Serre, \emph{Multi-dimensional {H}yperbolic {P}artial
  {D}ifferential {E}quations}, Oxford University Press, 2007.

\bibitem[BN07]{BN07}
Didier Bresch and Pascal Noble, \emph{Mathematical justification of a shallow
  water model}, Methods Appl. Anal. \textbf{14} (2007), no.~2, 87--117.

\bibitem[BP91]{BP91}
Christine Bernardi and Olivier Pironneau, \emph{On the shallow water equations
  at low {R}eynolds number}, Comm. Partial Differential Equations \textbf{16}
  (1991), no.~1, 59--104.

\bibitem[BPSTT]{BPSTT}
A.~Bousquet, M.~Petcu, M.-C. Shiue, R.~Temam, and J.~Tribbia, \emph{Boundary
  conditions for limitied area models}, Communications in Computational
  Physics, to appear.

\bibitem[BR11]{BR11}
Didier Bresch and Michael Renardy, \emph{Well-posedness of two-layer
  shallow-water flow between two horizontal rigid plates}, Nonlinearity
  \textbf{24} (2011), no.~4, 1081--1088.

\bibitem[Bre09]{Bre09}
Didier Bresch, \emph{Shallow-water equations and related topics}, Handbook of
  differential equations: evolutionary equations. {V}ol. {V}, Handb. Differ.
  Equ., Elsevier/North-Holland, Amsterdam, 2009, pp.~1--104.

\bibitem[CF48]{CF48}
R.~Courant and K.-O. ~Friedrichs,  \emph{Supersonic {F}low and {S}hock {W}aves}, Interscience Publishers, Inc., New York, N. Y. 1948, xvi+464 pp


\bibitem[CST10]{CST10}
Q.~Chen, M.-C. Shiue, and R.~Temam, \emph{The barotropic mode for the primitive
  equations, {S}peical issue in momory of {D}avid {G}ottlieb}, Journal of
  Scientific Computing \textbf{45} (2010), 167--199.

\bibitem[EN00]{EN00}
K.-J. Engel and R.~Nagel, \emph{{O}ne-{P}arameter {S}emigroups for {L}inear
  {E}volution {E}quations}, Graduate Texts in Math., vol. 194, Springer-Verlag,
  2000.

\bibitem[Gri85]{Gri85}
P.~Grisvard, \emph{Elliptic {P}roblems in {N}onsmooth {D}omains}, Monographs
  and Studies in Mathematics, Pitman, Boston, 1985.

\bibitem[Hor65]{Hor65}
L.~H$\ddot{\text{o}}$rmander, \emph{${L}^2$ estimates and existence theorems
  for the $\bar\partial$ operator}, Acta Math. \textbf{113} (1965), 89--152.

\bibitem[HPT11]{HPT11}
A.~Huang, M.~Petcu, and R.~Temam, \emph{The one-dimensional supercritical shallow-water equations with
  topography}, Annals of the University of Bucharest (Mathematical Series)
  \textbf{2 (LX)} (2011), 63--82.

\bibitem[HPT12]{HPT12}
\bysame, \emph{The nonlinear 2d supercritical inviscid
  shallow water equations in a rectangle}, submitted.

\bibitem[HP74]{HP74}
E.~Hille and R.S. Phillips, \emph{Functional {A}nalysis and {S}emi-{G}roups},
  American Mathematical Society, Providence, RI, 1974 (Third printing of the
  revised edition of 1957, AMS Colloquium Publications, vol. XXXI).

\bibitem[HT12b]{HT12b}
A.~Huang and R.~Temam, \emph{The nonlinear 2d subcritical inviscid shallow
  water equations with periodicity in one direction}, in preparation (under
  completion).

\bibitem[Kre70]{Kre70}
H.-O. Kreiss, \emph{Initial boundary value problems for hyperbolic systems},
  Comm. Pure Appl. Math \textbf{23} (1970), 277--298.

\bibitem[KT80]{KT80}
K.~Kojima and M.~Taniguchi, \emph{Mixed problem for hyperbolic equations in a
  domain with a corner}, Funkcialaj Ekvacioj \textbf{23} (1980), 171--195.

\bibitem[KU82]{KU82}
Hiroshi Kanayama and Teruo Ushijima, \emph{On the viscous shallow-water
  equations. {I}. {D}erivation and conservation laws}, Mem. Numer. Math.
  (1982), no.~8-9, 39--64.

\bibitem[KU88]{KU88}
\bysame, \emph{On the viscous shallow-water equations. {II}.\ {A} linearized
  system}, Bull. Univ. Electro-Comm. \textbf{1} (1988), no.~2, 347--355.

\bibitem[KU89]{KU89}
\bysame, \emph{On the viscous shallow-water equations. {III}. {A} finite
  element scheme}, Bull. Univ. Electro-Comm. \textbf{2} (1989), no.~1, 47--62.

\bibitem[Lop70]{Lop70}
Ya.~B. Lopatinskii, \emph{The mixed {C}auchy-{D}irichlet type problem for
  equations of hyperbolic type}, Dopovfdf Akad. Nauk Ukrai''n. RSR Ser. A
  \textbf{668} (1970), 592--594.

\bibitem[Ore95]{Ore95}
P.~Orenga, \emph{Un th$\acute{\text{e}}$or$\grave{\text{e}}$me d'existence de
  solutions d'un probl$\acute{\text{e}}$me de shallow water}, Archive for
  Rational Mechanics and Analysis \textbf{130} (1995), 183--204.

\bibitem[OS78]{OS78}
J.~Oliger and A.~Sundstr$\ddot{\text{o}}$m, \emph{Theoretical and practical
  aspects of some initial-boundary value problems in fluid dynamics}, SIAM J.
  Appl. Math. \textbf{35 (3)} (1978), 419--446.

\bibitem[Osh73]{Osh73}
Stanley Osher, \emph{Initial-boundary value problems for hyperbolic systems in
  regions with corners. {I}}, Trans. Amer. Math. Soc. \textbf{176} (1973),
  141--165.

\bibitem[Osh74]{Osh74}
\bysame, \emph{Initial-boundary value problems for hyperbolic systems in
  regions with corners. {II}}, Trans. Amer. Math. Soc. \textbf{198} (1974),
  155--175.

\bibitem[PT11]{PT11}
M.~Petcu and R.~Temam, \emph{The one-dimensional shallow water equations with
  transparent boundary conditions}, Math. Meth. Appl. Sci. (2011), DOI:
  10.1002/mma.1482.

\bibitem[RM74]{RM74}
Jeffrey~B. Rauch and Frank~J. Massey, III, \emph{Differentiability of solutions
  to hyperbolic initial-boundary value problems}, Trans. Amer. Math. Soc.
  \textbf{189} (1974), 303--318. \MR{0340832 (49 \#5582)}

\bibitem[RTT08b]{RTT08b}
A.~Rousseau, R.~Temam, and J.~Tribbia, \emph{The 3{D} {P}rimitive {E}quations
  in the absence of viscosity: {B}oundary conditions and well-posedness in the
  linearized case}, J. Math. Pures Appl. \textbf{89} (2008), 297--319.

\bibitem[Rud91]{Rud91}
W.~Rudin, \emph{Functional {A}nalysis}, {S}econd ed., International Series in
  Pure and Applied Mathematics, McGraw-Hill Inc., New York, 1991.

\bibitem[SLTT]{SLTT}
M.-C. Shiue, J.~Laminie, R.~Temam, and J.~Tribbia, \emph{Boundary value
  problems for the shallow water equations with topography}, Journal of
  Geophysical Research-Oceans \textbf{116} (2011), C02015.

\bibitem[Tan78]{Tan78}
Masaru Taniguchi, \emph{Mixed problem for wave equation in the domain with a
  corner}, Funkcialaj Ekvacioj \textbf{21} (1978), 249--259.

\bibitem[TT03]{TT03}
R.~Temam and J.~Tribbia, \emph{Open boundary conditions for the primitive and
  {B}oussinesq equations}, J. Atmospheric Sci. \textbf{60 (21)} (2003),
  2647--2660.

\bibitem[Ush83]{Ush83}
Teruo Ushijima, \emph{A semigroup theoretical analysis of a finite element
  method for a linearized viscous shallow-water system}, Publ. Res. Inst. Math.
  Sci. \textbf{19} (1983), no.~3, 1305--1328.

\bibitem[WPT97]{WPT97}
T.~Warner, R.~Peterson, and R.~Treadon, \emph{A tutorial on lateral boundary
  conditions as a basic and potentially serious limitation to regional
  numerical weather prediction}, Bull. Amer. Meteor. Soc. (1997), 2599--2617.

\bibitem[Yos80]{Yos80}
K.~Yosida, \emph{Functional {A}nalysis}, {S}ixth ed., Springer-Verlag, Berlin,
  1980.

\end{thebibliography}
\newcommand{\etalchar}[1]{$^{#1}$}
\providecommand{\bysame}{\leavevmode\hbox to3em{\hrulefill}\thinspace}
\providecommand{\MR}{\relax\ifhmode\unskip\space\fi MR }
\providecommand{\MRhref}[2]{%
  \href{http://www.ams.org/mathscinet-getitem?mr=#1}{#2}
}
\providecommand{\href}[2]{#2}

\end{document}